\documentclass{amsart}

\usepackage{amsmath}
\usepackage{amsthm}
\usepackage{amsfonts}
\usepackage{amssymb}
\usepackage{enumerate}
\usepackage{graphicx}
\usepackage{color}
\usepackage{float}

\newcommand{\E}{\mathbb{E}}
\newcommand{\Prob}{\mathbb{P}}

\newcommand{\R}{\mathbb{R}}

\newcommand{\eps}{\varepsilon}

\theoremstyle{plain}
  \newtheorem{theorem}{Theorem}

  \newtheorem{proposition}[theorem]{Proposition}
  \newtheorem{fact}[theorem]{Fact}
  \newtheorem{lemma}[theorem]{Lemma}
  \newtheorem{corollary}[theorem]{Corollary}
  
  \newtheorem{question}[theorem]{Question}

\theoremstyle{definition}
  \newtheorem{definition}[theorem]{Definition}
  \newtheorem{example}[theorem]{Example}
  \newtheorem{remark}[theorem]{Remark}

\begin{document}
\title[Random perturbation of low rank matrices]{Random perturbation of low rank matrices: Improving  classical bounds}

\author[S. O'Rourke]{Sean O'Rourke}
\thanks{S. O'Rourke is supported by grant AFOSAR-FA-9550-12-1-0083.}
\address{Department of Mathematics, University of Colorado at Boulder, Boulder, CO 80309 }
\email{sean.d.orourke@colorado.edu}

\author{Van Vu}
\thanks{V. Vu is supported by research grants DMS-0901216 and AFOSAR-FA-9550-09-1-0167.}
\address{Department of Mathematics, Yale University, PO Box 208283, New Haven , CT 06520-8283, USA}
\email{van.vu@yale.edu}

\author{Ke Wang}
\address{Institute for Mathematics and its Applications, University of Minnesota, Minneapolis, MN  55455, USA}
\email{wangk@umn.edu}

\begin{abstract}
Matrix perturbation inequalities, such as Weyl's theorem  (concerning the singular values) and the Davis-Kahan  theorem (concerning the singular vectors), play essential roles in  quantitative science; in particular, these bounds have found application in data analysis as well as related areas of engineering and computer science.

In many situations, the perturbation is assumed to be random, and the original matrix  has certain  structural properties (such as having low rank). We show that, in this scenario,  classical perturbation results, such as Weyl and Davis-Kahan, can be improved significantly.  We  believe many of  our new bounds are close to optimal and also discuss some applications. 
 \end{abstract}
 
 

\maketitle

\section{Introduction}

The singular value decomposition of a real $m \times n$ matrix $A$ is a factorization of the form $A = U \Sigma V^\mathrm{T}$, where $U$ is a $m \times m$ orthogonal matrix, $\Sigma$ is a $m \times n$ rectangular diagonal matrix with non-negative real numbers on the diagonal, and $V^\mathrm{T}$ is an $n \times n$ orthogonal matrix.  The diagonal entries of $\Sigma$ are known as the \emph{singular values} of $A$.  The $m$ columns of $U$ are the \emph{left-singular vectors} of $A$, while the $n$ columns of $V$ are the \emph{right-singular vectors} of $A$.  If $A$ is symmetric, the singular values are given by the absolute value of the eigenvalues, and the singular vectors can be expressed in terms of the eigenvectors of $A$.  Here, and in the sequel, whenever we write \emph{singular vectors}, the reader is free to interpret this as left-singular vectors or right-singular vectors provided the same choice is made throughout the paper.  

An important problem in statistics and numerical analysis is to compute the first $k$ singular values and vectors of an $m \times n$ matrix $A$.  In particular, the largest few singular values and corresponding singular vectors are typically the most important.  Among others, this problem lies at the heart of Principal Component Analysis (PCA), which has a very wide range of applications (for many examples, see \cite{KVbook, LR} and the references therein) and in the closely related low rank approximation procedure often used in theoretical computer science and combinatorics.   In application, the dimensions $m$ and $n$ are typically large and $k$ is small, often a fixed constant. 
 
 \subsection{The perturbation problem}
A problem of fundamental importance in quantitative science (including pure  and applied mathematics, statistics, engineering, and computer science) is to estimate how a small perturbation to the data effects the singular values and singular vectors. This problem has been discussed  in virtually every  text book on quantitative linear algebra and numerical analysis  (see, for instance,  \cite{BT, Hig1, Hig2, SS}), and is the main focus of this paper.  

We model the problem as follows.  Consider a real (deterministic) $m \times n$ matrix $A$ with singular values 
$$\sigma_1 \geq \sigma_2  \geq \cdots \geq \sigma_{\min\{m,n\}} \geq 0$$
and  corresponding  singular vectors  $v_1, v_2, \ldots, v_{\min\{m,n\}}.$ We will call $A$ the data matrix.  In general, the vector $v_i$ is not unique.  However, if $\sigma_i$ has multiplicity one, then $v_i$ is determined up to sign.  Instead of $A$, one often needs to work with  $A+E$, where $E$ represents the perturbation matrix.  Let 
$$ \sigma_1' \geq \cdots \geq \sigma_{\min\{m,n\}}' \geq 0 $$ 
denote the singular values of $A+E$ with corresponding singular vectors $v_1', \ldots, v_{\min\{m,n\}}'$.  In this paper, we address the following two questions.  

\begin{question}
	When is $v_i'$ a good approximation of $v_i$?
\end{question}  
\begin{question} \label{quest:weyl}
	When is $\sigma_i'$ a good approximation of $\sigma_i$?  
\end{question}

These two questions are classically addressed by the Davis-Kahan-Wedin sine theorem and Weyl's inequality.  
Let us begin with the first question in the case when $i=1$.  A canonical  way (coming from the numerical analysis literature; see for instance  \cite{GVL})  
to measure the distance between two unit vectors $v$ and $v'$ is to look at $ \sin \angle(v,v')$,
where $\angle(v,v')$ is the angle between $v$ and $v'$ taken in $[0,\pi/2]$. 
It has been observed by numerical analysts (in the setting where $E$ is deterministic) for quite some time  that the key parameter to consider in the bound is the gap (or separation) $\sigma_1 - \sigma_2'$.  The first result in this direction is the famous Davis-Kahan sine $\theta$ theorem \cite{DK} for Hermitian matrices. 
A version for the singular vectors was proved later by Wedin \cite{W}.  

Throughout the paper, we use $\|M\|$ to denote the spectral norm of a matrix $M$.  That is, $\|M\|$ is the largest singular value of $M$. 

\begin{theorem}[Davis-Kahan, Wedin; sine theorem; Theorem V.4.4 from \cite{SS}] \label{thm:wedin}
\begin{equation}\label{eq:DK}
	\sin \angle(v_1,v_1') \leq \frac{\|E\|}{\sigma_1-\sigma_2'}.
\end{equation}
\end{theorem} 

In certain cases, such as when $E$ is random, it is more natural to deal with the gap 
\begin{equation} \label{def:delta}
	\delta := \sigma_1 - \sigma_2,
\end{equation} 
between the first and second singular values of $A$ instead of $\sigma_1 - \sigma_2'$.  In this case, Theorem \ref{thm:wedin} implies the following bound.   

\begin{theorem}[Modified sine theorem] \label{thm:modified:wedin}
$$ \sin \angle(v_1, v_1') \leq 2 \frac{\|E\|}{\delta}. $$
\end{theorem} 

\begin{remark}
Theorem \ref{thm:modified:wedin} is trivially true when $\delta \leq 2 \|E\|$ since sine is always bounded above by one.  In other words, even if the vector $v_1'$ is not uniquely determined, the bound is still true for any choice of $v_1'$.  On the other hand, when $\delta > 2 \|E\|$, the proof of Theorem \ref{thm:modified:wedin} reveals that the vector $v_1'$ is uniquely determined up to sign. 
\end{remark}
 
As the next example shows, the bound in Theorem \ref{thm:modified:wedin} is sharp, up to the constant $2$.
 
\begin{example}
Let $0 < \eps < 1/2$, and take
$$ A := \begin{pmatrix} 1 + \eps & 0 \\ 0 & 1 - \eps \end{pmatrix}, \quad E := \begin{pmatrix} - \eps & \eps \\ \eps & \eps \end{pmatrix}. $$
Then $\sigma_1 = 1 + \eps$, $\sigma_2 = 1 - \eps$ with $v_1 = (1, 0)^\mathrm{T}$ and $v_2 = (0,1)^\mathrm{T}$.  Hence, $\delta = 2 \eps$.  In addition, 
$$ A + E = \begin{pmatrix} 1 & \eps \\ \eps & 1 \end{pmatrix}, $$
and a simple computation reveals that
$\sigma_1' = 1 + \eps$, $\sigma_2' = 1 - \eps$ but $v_1' = (1/\sqrt{2}, 1/\sqrt{2})^{\mathrm{T}}$ and $v_2' = (1/\sqrt{2}, - 1/\sqrt{2})^\mathrm{T}$.  Thus,
$$ \sin \angle(v_1, v_1') = \frac{1}{\sqrt{2}} = \frac{ \|E \|}{\delta} $$
since $\|E\| = \sqrt{2} \eps$.  
\end{example}

More generally, one can consider approximating the $i$-th singular vector $v_i$ or the space spanned by the first $i$ singular vectors $\mathrm{Span}\{v_1, \ldots, v_i \}$.  Naturally, in these cases, a version of Theorem \ref{thm:modified:wedin} requires one to consider the gaps
$$ \delta_i := \sigma_i  - \sigma_{i+1}; $$
see Theorems \ref{thm:modified:wedin3} and \ref{thm:modified:wedin2} below for details.  

Question \ref{quest:weyl} is addressed by Weyl's inequality.  In particular, Weyl's perturbation theorem \cite{Wy} gives the following deterministic bound for the singular values (see \cite[Theorem IV.4.11]{SS} for a more general perturbation bound due to Mirsky \cite{M}).  

\begin{theorem} [Weyl's bound]  \label{theorem:Weyl} 
\begin{equation*}  
	\max_{1 \leq i\leq \min\{m,n\}} | \sigma_i -\sigma_i'| \le \|E \|. 
\end{equation*} 
\end{theorem} 

For more discussions concerning general perturbation bounds, we refer the reader to \cite{B, SS} and references therein.  We now pause for a moment to prove Theorem \ref{thm:modified:wedin}.  

\begin{proof}[Proof of Theorem \ref{thm:modified:wedin}]
If $\delta \leq 2 \|E\|$, the theorem is trivially true since sine is always bounded above by one.  Thus, assume $\delta > 2 \|E\|$.  By Theorem \ref{theorem:Weyl}, we have
$$ \sigma_1' - \sigma_2' \geq \delta - 2 \|E\| > 0, $$
and hence the singular vectors $v_1$ and $v_1'$ are uniquely determined up to sign.  By another application of Theorem \ref{theorem:Weyl}, we obtain
$$ \delta = \sigma_1 - \sigma_2 \leq \sigma_1 - \sigma_2' + \|E\|. $$
Rearranging the inequality, we have
$$ \sigma_1 - \sigma_2' \geq \delta - \|E\| \geq \frac{1}{2} \delta > 0. $$
Therefore, by \eqref{eq:DK}, we conclude that
$$ \sin \angle(v_1, v_1') \leq \frac{\|E\|}{ \sigma_1 - \sigma_2'} \leq 2 \frac{\|E\|}{ \delta}, $$
and the proof is complete.  
\end{proof}

\subsection{The random setting}
Let us now focus on the matrices  $A$ and $E$. 
It has become common practice to assume that the perturbation matrix $E$ is random. 
Furthermore,  researchers have observed that 
data matrices are usually  not arbitrary.  They  often possess  certain  structural properties. Among these properties, one of the most frequently seen is having low rank (see, for instance, \cite{CP, CR, CRT, CS, TK} and references therein).  

The goal in this paper is to show that in this situation, one can significantly improve classical results like Theorems \ref{thm:modified:wedin} and \ref{theorem:Weyl}.
To give a quick example, let us assume  that  $A$ and $E$ are $n \times n$ matrices and that $E$ is a random Bernoulli matrix, i.e., its entries are independent and identically distributed (iid) random variables that take values $\pm 1$ with probability $1/2$.  It is well known that in this case 
 $\|E\|= (2+o(1)) \sqrt n $ with high probability\footnote{We use asymptotic notation under the assumption that $n \to \infty$.  Here we use $o(1)$ to denote a term which tends to zero as $n$ tends to infinity.} \cite[Chapter 5]{BS}.  Thus, the above two theorems imply the following. 
\begin{corollary} \label{wedin-cor} 
If $E$ is an $n \times n$ Bernoulli\footnote{More generally, Corollary \ref{wedin-cor} applies to a large class of random matrices with independent entries.  Indeed, the results in \cite[Chapter 5]{BS} and hence Corollary \ref{wedin-cor} hold when $E$ is any $n \times n$ random matrix whose entries are iid random variables with zero mean, unit variance (which is just a matter of normalization), and bounded fourth moment.  } random matrix, then, for any $\eta > 0$, with probability $1-o(1)$,
\begin{equation*}  \max_{1 \leq i \leq n} |\sigma_i  -\sigma _i'| \le (2 + \eta) \sqrt n, \end{equation*} 
and
\begin{equation} \label{bound0}  \sin \angle(v_1, v_1') \leq 2 (2+\eta) \frac{\sqrt n }{\delta}. \end{equation}
\end{corollary}

Among others, this shows that we must have $\delta > 2 (2 + \eta) \sqrt{n}$ in order for the bound in \eqref{bound0} to be nontrivial.  It turns out that the bounds in Corollary \ref{wedin-cor} are far from being sharp.  Indeed, we present the results of a numerical simulation for $A$ being a $n \times n$ matrix of rank 2 when $n=400$, $\delta=8$, and where $E$ is a random Bernoulli matrix.  
It is easy to see that for the parameters $n=400$ and $\delta =8$,   Corollary \ref{wedin-cor} does not give  a useful bound
(since $\frac{\sqrt{n}}{\delta} = 2.5 >1$). However, Figure \ref{young} shows that,  with high probability, $\sin \angle(v_1,v_1') \leq 0.2$, which means 
$v_1'$ approximates $v_1$ with a relatively small error.  Our main results attempt to address this inefficiency in the Davis-Kahan-Wedin and Weyl bounds and provide sharper bounds than those given in Corollary \ref{wedin-cor}.  As a concrete example, in the case when $E$ is a random Bernoulli matrix, our results imply the following bounds.

\begin{figure}[!t]
 \begin{center}
   \includegraphics[width=9cm]{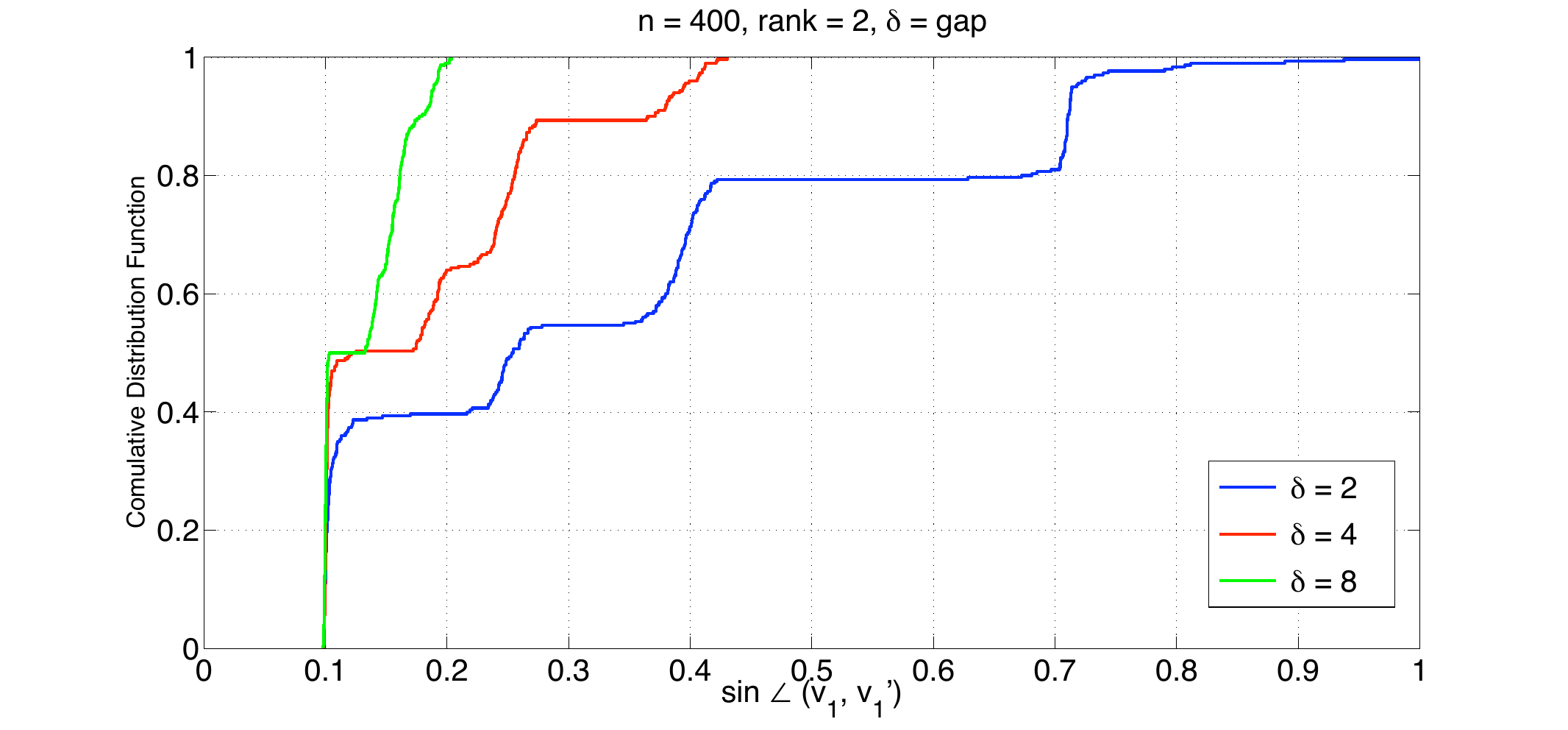}
   \includegraphics[width=9cm]{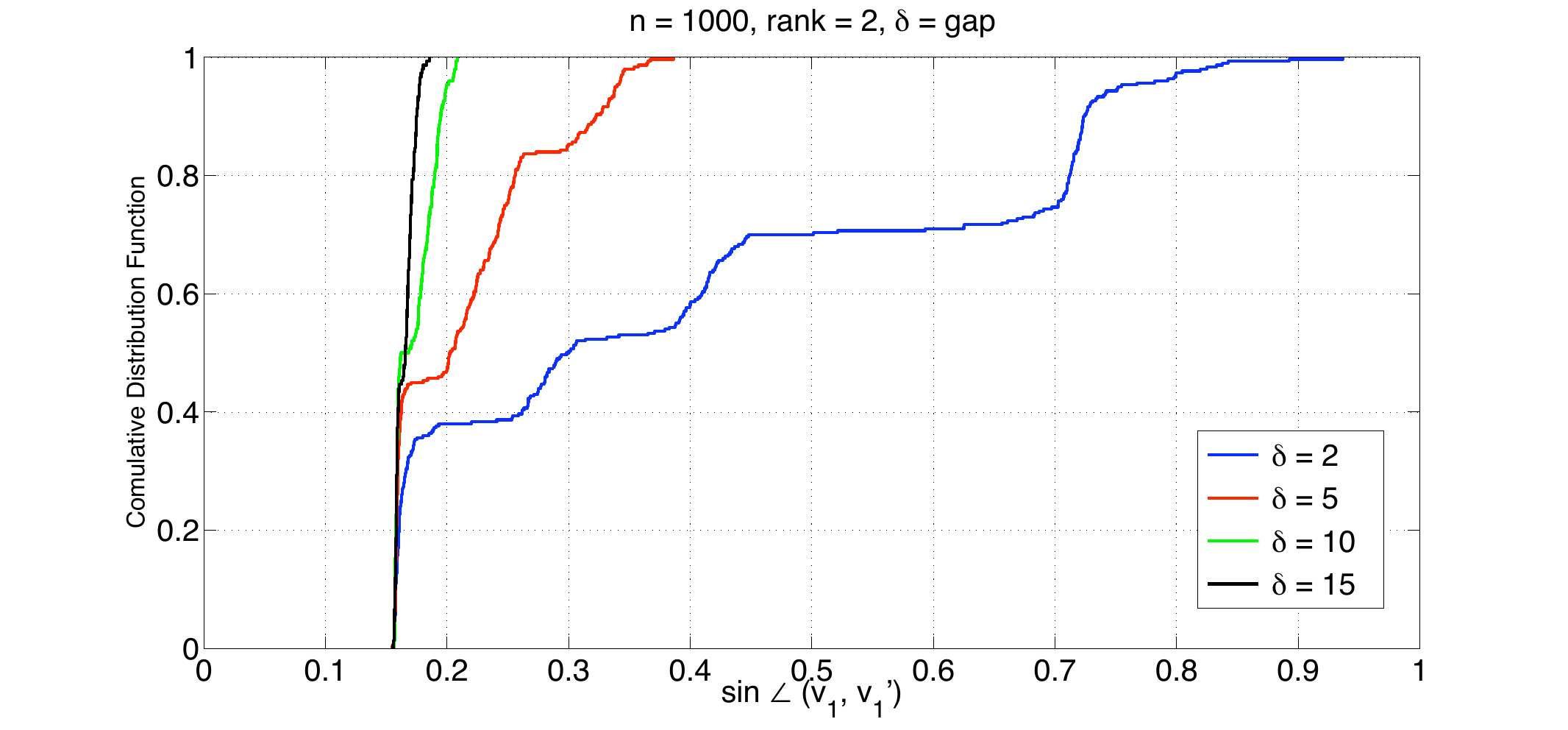}
   \caption{The cumulative distribution functions of $\sin \angle(v_1, v_1')$ where $A$ is a $n \times n$ deterministic matrix with rank $2$ ($n=400$ for the figure on top and $n=1000$ for the one below) and the noise $E$ is a Bernoulli random matrix, evaluated from $400$ samples (top figure) and $300$ samples (bottom figure). In both figures, the largest singular value of $A$ is taken to be $200$.}
   \label{young}
 \end{center}
\end{figure}

 \begin{theorem} \label{theorem:main1} Let $ E$ be a $n \times n$ Bernoulli random matrix, and let $A$ be a $n \times n$ matrix with rank $r$.  For every $\varepsilon > 0$ there exists constants $C_0, \delta_0 > 0$ (depending only on $\varepsilon$) such that if $\delta \geq \delta_0$ and $\sigma_1 \geq \max\{n,\sqrt{n}\delta\}$, then, with probability at least $1-\varepsilon$,
\begin{equation*} \sin \angle (v_1, v_1') \leq C \frac{\sqrt{r}} {\delta}  . 
\end{equation*}  
 \end{theorem}
 
 

\begin{theorem} \label{thm:probweyl0}
 Let $E$ be an $n \times n$  Bernoulli random matrix, and let  $A$ be an $n \times n$ matrix with rank $r$ satisfying  $ \sigma_1 \geq n$. For every $\eps > 0$, there exists a constant $C_0>0$ (depending only on $\varepsilon$) 
such that, with probability at least $1-\varepsilon$, 
\begin{equation*} \label{eq:probweylbnd0}
\sigma_1  -  C  \leq \sigma_1'  \leq  \sigma_1  +  C  \sqrt{r}.
\end{equation*}
\end{theorem}

In particular, when the rank $r$ is significantly smaller than $n$, the bounds in Theorems \ref{theorem:main1} and \ref{thm:probweyl0} are significantly better than those appearing in Corollary \ref{wedin-cor}.  
The intuition behind Theorems \ref{theorem:main1} and \ref{thm:probweyl0} comes from the following heuristic of the second author.



\begin{quote}
If $A$ has rank $r$, all actions of $A$ focus on an $r$ dimensional subspace; intuitively then, $E$ must act like an  $r$ dimensional random matrix rather than an $n$ dimensional one. 
\end{quote}

  This means that  the {\it real dimension} of the problem  is $r$, not $n$. 
While it is clear that one cannot automatically  ignore the (rather wild) action of $E$ outside the range of $A$, this intuition, if true, explains the appearance of the $\sqrt{r}$ factor in the bounds of Theorems \ref{theorem:main1} and \ref{thm:probweyl0} instead of the $\sqrt{n}$ factor appearing in Corollary \ref{wedin-cor}.  


%
%
%
%
%

While Theorems \ref{theorem:main1} and \ref{thm:probweyl0} are stated only for Bernoulli random matrices $E$, our main results actually hold under very mild assumptions on $A$ and $E$.  As a matter of fact, in the strongest results, we will not even need the entries of $E$ to be independent.

 \subsection{Preliminaries: Models of random noise} 
We now state the assumptions we require for the random matrix $E$.  While there  are many models of random matrices, we can capture almost all natural  models   by focusing on a common  property.

\begin{definition} \label{def:concentration}
We say the  $m \times n$  random matrix $E$ is $(C_1,c_1, \gamma)$-concentrated if for all unit vectors $u \in \mathbb{R}^m, v \in \mathbb{R}^n$, and every $t>0$, 
\begin{equation} \label{eq:concentration}
	\Prob( |u^T E v| > t ) \leq C_1 \exp(-c_1 t^\gamma).
\end{equation}	
\end{definition}

The key parameter is $\gamma$. It is easy to verify the following fact, which asserts that the concentration property is closed under addition. 
\begin{fact}  \label{fact1}  If $E_1$ is $(C_1,c_1, \gamma)$-concentrated and $E_2$ is $(C_2,c_2, \gamma)$-concentrated, then 
$E_3 =E_1+E_2$ is $(C_3, c_3, \gamma)$-concentrated for some $C_3, c_3$ depending on $C_1,c_1, C_2, c_2$. \end{fact}

Furthermore, the concentration property guarantees a bound on $\| E \|$.  A standard net argument (see Lemma \ref{lemma:net}) shows

\begin{fact}  \label{fact2}  If $E$ is $(C_1,c_1, \gamma)$-concentrated  then there are constants $C',  c' >0$ such that $\Prob (\| E \| \ge C' n^{1/\gamma} )  \le C_1 \exp (-c'n) $. \end{fact}

For readers not familiar with random matrix theory, let us point out why the concentration property is expected to hold for many natural models. If $E$ is random and $v$ is fixed, then the vector $Ev$ must look random. It is well known that in a high dimensional space, a random isotropic vector, with very high probability, is nearly orthogonal to any fixed vector. Thus, one expects that very likely, the inner product of $u$ and $E v$ is small. Definition \ref{def:concentration} is a way to express this observation quantitatively. 

It turns out that all random matrices with independent entries satisfying a mild condition have the concentration property.  Indeed, if $E_{ij}$ denotes the $(i,j)$-entry of $E$ and the entries of $E$ are assumed to be independent, then the bilinear form
$$ u^\mathrm{T} E v = \sum_{i=1}^m \sum_{j=1}^n u_i E_{ij} v_j $$
is just a sum of independent random variables.  If, in addition, the entries of $E$ have mean zero, then, by linearity, $u^\mathrm{T} E v$ also has mean zero.  Hence, \eqref{eq:concentration} can be viewed as a concentration inequality, which expresses how the sum of independent random variables deviates from its mean.  With this interpretation in mind, many models of random matrices can be shown to satisfy \eqref{eq:concentration}.  
 In particular, Lemma \ref{lemma:bernoulli} shows that if $E$ is a $n \times n$ Bernoulli random 
 matrix, then $E$ is $\left(2, \frac{1}{2}, 2 \right)$-concentrated, and $\|E\| \leq 3 \sqrt{n}$ with high probability \cite{V,Vnorm}. 
 However, a convenient feature of  the definition is that   independence  between the entries is not a requirement.
  For instance, it is easy to show that a random orthogonal matrix 
 satisfies the concentration property.  We continue the discussion of the $(C_1,c_1,\gamma)$-concentration property (Definition \ref{def:concentration}) in Section \ref{sec:concentration}.

\section{Main results}

We now state our main results.  We begin with an extension of   Theorem \ref{theorem:main1}.

\begin{theorem} \label{thm:main}
Assume that $E$ is $(C_1, c_1, \gamma)$-concentrated for a trio of constants $C_1, c_1, \gamma >0$, and suppose $A$ has rank $r$.  Then, for any $t>0$, 
\begin{equation} \label{eq:bnd:main}
	\sin \angle(v_1,v_1') \leq  4 \sqrt{2} \left( \frac{t r^{1/\gamma}} {\delta}  +     \frac{ \|E\|}{\sigma_1} +  \frac{ \|E\|^2}{\sigma_1 \delta} \right) 
\end{equation}
with probability at least 
\begin{equation} \label{eq:bnd:main:prob}
	1 - 54 C_1 \exp\left(-c_1\frac{\delta^\gamma}{8^\gamma} \right) - 2C_1 9^{2r} \exp \left( -c_1 r \frac{t^\gamma}{4^{\gamma}} \right). 
\end{equation}
\end{theorem}

\begin{remark} \label{remark:boundonE}
Using  Fact \ref{fact2}, one can replace $\| E\|$ on the right-hand side of \eqref{eq:bnd:main} by $C' n^{1/\gamma }$, which yields that  
$$ \sin \angle(v_1,v_1') \leq 4 \sqrt{2}  \left( \frac{t r^{1/\gamma}} {\delta} +   \frac{ C' n^{1/\gamma} }{\sigma_1} +  \frac{ C'^2 n^{2/\gamma} }{\sigma_1 \delta} \right)$$
with probability at least 
$$ 1 - 54 C_1 \exp\left(-c_1\frac{\delta^\gamma}{8^\gamma} \right) - 2C_1 9^{2r} \exp \left( -c_1 r \frac{t^\gamma}{4^{\gamma}} \right) - C_1 \exp( -c' n). $$  
However, we prefer to state our theorems in the form of Theorem \ref{thm:main}, as the bound $C' n^{1/\gamma}$, in many cases, may not be optimal. 
\end{remark} 

Because Theorem \ref{thm:main} is stated in such generality, the bounds can be difficult to interpret.  For example, it is not completely obvious when the probability in \eqref{eq:bnd:main:prob} is close to one.  Roughly speaking, the two error terms in the probability bound are controlled by the gap $\delta$ and the parameter $t$ (which can be taken to be any positive value).  Specifically, the first term
\begin{equation} \label{eq:term:error1}
	54 C_1 \exp\left(-c_1\frac{\delta^\gamma}{8^\gamma} \right) 
\end{equation}
goes to zero as $\delta$ gets larger, and the second term 
\begin{equation} \label{eq:term:error2}
	2C_1 9^{2r} \exp \left( -c_1 r \frac{t^\gamma}{4^{\gamma}} \right) 
\end{equation}
goes to zero as $t$ tends to infinity.  As a consequence, we obtain the following immediate corollary of Theorem \ref{thm:main} (and Lemma \ref{lemma:conc-nonsym}) in the case when the entries of $E$ are independent.  

\begin{corollary} \label{cor:main1}
Assume that $E$ is an $m \times n$ random matrix with independent entries which have mean zero and are bounded almost surely in magnitude by $K$ for some $K > 0$.  Suppose $A$ has rank $r$.  Then for every $\eps > 0$, there exists $C_0, c_0, \delta_0 > 0$ (depending only on $\eps$ and $K$) such that if $\delta \geq \delta_0$, then
\begin{equation} \label{eq:corbnd}
	\sin \angle(v_1,v_1') \leq C_0 \left( \frac{ \sqrt{r} }{\delta} + \frac{ \|E \|}{\sigma_1} + \frac{ \| E \|^2 }{ \sigma_1 \delta } \right) 
\end{equation}
with probability at least $1 - \eps$.  
\end{corollary}

The first term $\frac{\sqrt{r}} {\delta} $  on the right-hand side of \eqref{eq:corbnd} is precisely the conjectured optimal bound coming from the intuition discussed above. The second term $\frac{\|E\|}{ \sigma_1}$ is necessary. 
If $\|E \| \gg \sigma_1$, then  the intensity of the noise is much stronger than the strongest  signal in the data matrix, so $E$ would corrupt  $A$ completely. 
Thus in order to retain crucial information about $A$, it seems  necessary  to assume $\|E\| < \sigma_1$.  We are not absolutely sure about the necessity of the third term 
$ \frac{ \|E\|^2}{\sigma_1 \delta}$, but under the condition $\|E\| \ll \sigma_1 $, this term is superior to the Davis-Kahan-Wedin bound $\frac{\|E\| }{ \delta}$ appearing in Theorem \ref{thm:modified:wedin}.  

While it remains an open question to determine whether the bounds in Theorem \ref{thm:main} are optimal, we do note that in certain situations the bounds are close to optimal.  Indeed, in \cite{BGN}, the eigenvectors of perturbed random matrices are studied, and, under various technical assumptions on the matrices $A$ and $E$, the results in \cite{BGN} give the exact asymptotic behavior of the dot product $|v_1 \cdot v_1'|$.  Rewriting the dot product in terms of cosine (and further expressing the value in terms of sine), we find that the bounds in \eqref{eq:bnd:main} match the exact asymptotic behavior obtained in \cite{BGN}, up to constant factors.  Similar results in \cite{OW} also match the bound in \eqref{eq:bnd:main}, up to constant factors, in the case when $E$ is a Wigner random matrix and $A$ has rank one.  

Corollary \ref{cor:main1} provides a bound which holds with probability at least $1 - \eps$.  As another consequence of Theorem \ref{thm:main}, we obtain the following bound which holds with probability converging to $1$.  

\begin{corollary} \label{cor:main2}
Assume that $E$ is an $m \times n$ random matrix with independent entries which have mean zero and are bounded almost surely in magnitude by $K$ for some $K > 0$.  Suppose $A$ has rank $r$.  Then there exists $C_0 > 0$ (depending only on $K$) such that if $\alpha_n$ is any sequence of positive values converging to infinity and $\delta \geq \alpha_n$, then
\[ \sin \angle(v_1,v_1') \leq C_0 \left( \frac{ \alpha_n \sqrt{r} }{\delta} + \frac{ \|E \|}{\sigma_1} + \frac{ \| E \|^2 }{ \sigma_1 \delta } \right) \]
with probability $1 - o(1)$.  Here, the rate of convergence implicit in the $o(1)$ notation depends on $K$ and $\alpha_n$.  
\end{corollary}

Before continuing, we pause to make one final remark regarding Corollaries \ref{cor:main1} and \ref{cor:main2}.  In stating our main results below, we will always state them in the generality of Theorem \ref{thm:main}.  However, each of the results can be specialized in several different directions similar to what we have done in Corollaries \ref{cor:main1} and \ref{cor:main2}.  In the interest of space, we will not always state all such corollaries.

We are able to  extend Theorem \ref{thm:main} in two different ways. First, we can bound the angle between $v_j$ and $v_j'$ for any index $j$. Second,  and more importantly, we can bound the 
angle between the subspaces spanned by $\{v_1, \dots, v_j \}$ and $\{v_1', \dots, v_j' \}$, respectively. As the projection onto the subspaces spanned by the first few singular vectors  (i.e.,  low rank approximation)  plays an important role in  a vast collection of problems,  this result potentially has a large number of applications.  
 x

We begin by bounding the largest principal angle between
\begin{equation} \label{eq:uspan}
	V := \mathrm{Span}\{v_1, \ldots, v_j\} \quad \text{and}\quad V' := \mathrm{Span}\{v_1', \ldots, v_j'\}
\end{equation}
for some integer $1 \leq j \leq r$, where $r$ is the rank of $A$.  Let us recall that  if $U$ and $V$ are two subspaces of the same dimension, then the (principal) angle between them is defined as
\begin{equation} \label{eq:ssad}
	\sin \angle(U,V) := \max_{u \in U; u \neq 0} \min_{v \in V; v \neq 0} \sin \angle(u,v) = \|P_U - P_V \| = \|P_{U^\perp} P_{V} \|, 
\end{equation}
where $P_W$ denotes the orthogonal projection onto subspace $W$. 

\begin{theorem} \label{thm:subspace}
Assume that $E$ is $(C_1, c_1, \gamma)$-concentrated for a trio of constants $C_1, c_1, \gamma >0$.  Suppose $A$ has rank $r$, and let $1 \leq j \leq r$ be an integer.  Then, for any $t>0$,
\begin{equation} \label{eq:subspace:bnd}
	\sin \angle(V,V') \leq 4 \sqrt{2j} \left( \frac{t r^{1/\gamma}}{\delta_j} + \frac{\|E\|^2}{\sigma_j \delta_j} + \frac{ \|E\|}{\sigma_j} \right), 
\end{equation}
with probability at least 
\begin{equation} \label{eq:subspace:prob}
	1 - 6C_1 9^j \exp \left( -c_1 \frac{\delta_j^\gamma}{8^\gamma} \right) - 2C_1 9^{2r} \exp \left( -c_1 r \frac{t^\gamma}{4^\gamma} \right), 
\end{equation}
where $V$ and $V'$ are the $j$-dimensional subspaces defined in \eqref{eq:uspan}.
\end{theorem}

The error terms in \eqref{eq:subspace:prob} (as well as all other probability bounds appearing in our main results) can be controlled in a similar fashion as the error terms \eqref{eq:term:error1} and \eqref{eq:term:error2}.  Indeed, the first error term in \eqref{eq:subspace:prob} is controlled by the gap $\delta_j$ and the second term is controlled by the parameter $t$.  

We believe the factor of $\sqrt{j}$ in \eqref{eq:subspace:bnd} is suboptimal and is simply an artifact of our proof.  However, in many applications $j$ is significantly smaller than the dimension of the matrices, making the contribution from this term negligible. 

For comparison, we present an analogue of Theorem \ref{thm:modified:wedin}, which follows from the Davis-Kahan-Wedin sine theorem \cite[Theorem V.4.4]{SS}, using the same argument as in the proof of Theorem \ref{thm:modified:wedin}.
\begin{theorem}[Modified Davis-Kahan-Wedin sine theorem: singular space] \label{thm:modified:wedin3}
Suppose $A$ has rank $r$, and let $1 \leq j \leq r$ be an integer. Then, for an arbitrary matrix $E$, 
$$ \sin \angle (V, V') \leq 2\frac{\|E\|}{\delta_j}, $$
where $V$ and $V'$ are the $j$-dimensional subspaces defined in \eqref{eq:uspan}.
\end{theorem}

 It remains an open question to give an optimal version of Theorem \ref{thm:subspace} for subspaces corresponding to an arbitrary set of singular values.  However, we can use Theorem \ref{thm:subspace} repeatedly to obtain bounds for the case when one considers a few intervals of singular values.  For instance, by applying Theorem \ref{thm:subspace} twice, we  obtain the following result.  Denote $\delta_0 := \delta_1$.    

\begin{corollary} \label{cor:subspace}
Assume that $E$ is $(C_1, c_1, \gamma)$-concentrated for a trio of constants $C_1, c_1, \gamma >0$.  Suppose $A$ has rank $r$, and let $1 < j \leq l \leq r$ be integers.  Then, for any $t>0$,
\begin{equation} \label{eq:subspace2:bnd}
	\sin \angle(V,V') \leq 8 \sqrt{2l} \left( \frac{t r^{1/\gamma}}{\delta_{j-1}} + \frac{t r^{1/\gamma}}{\delta_l} + \frac{\|E\|^2}{\sigma_{j-1} \delta_{j-1}} + \frac{\|E\|^2}{\sigma_l \delta_l} +\frac{ \|E\|}{\sigma_l} \right), 
\end{equation}
with probability at least 
$$ 1 - 6C_1 9^{j-1} \exp \left( -c_1 \frac{\delta_{j-1}^\gamma}{8^\gamma} \right) - 6C_1  9^l \exp \left( -c_1 \frac{\delta_l^\gamma}{8^\gamma} \right) - 4C_1 9^{2r} \exp \left( -c_1 r \frac{t^\gamma}{4^\gamma} \right), $$ 
where
\begin{equation} \label{def:subspaceinterval}
	V:= \mathrm{Span}\{v_j,\ldots, v_l\} \quad \text{and}\quad V':=\mathrm{Span}\{v_j',\ldots,v_l'\}. 
\end{equation}
\end{corollary}
\begin{proof}
Let
\begin{align*}
	V_1 &:= \mathrm{Span}\{v_1,\ldots,v_l\}, \quad V_1' := \mathrm{Span}\{v_1',\ldots,v_l'\}, \\
	V_2 &:= \mathrm{Span}\{v_1,\ldots,v_{j-1}\}, \quad V_2' := \mathrm{Span}\{v_1',\ldots,v_{j-1}'\}.
\end{align*}
For any subspace $W$, let $P_W$ denote the orthogonal projection onto $W$.  It follows that $P_{W^\perp} = I - P_{W}$, where $I$ denotes the identity matrix.  By definition of the subspaces $V,V'$, we have
$$ P_V = P_{V_1} P_{V_2^\perp} \quad \text{and}\quad P_{V'} = P_{V_1'} P_{V_2'^\perp}. $$
Thus, by \eqref{eq:ssad}, we obtain
\begin{align*}
	\sin \angle(V,V') &= \| P_{V_1} P_{V_2^\perp} - P_{V_1'} P_{V_2'^\perp} \| \\
		&\leq \| P_{V_1} P_{V_2^\perp} - P_{V_1'} P_{V_2^\perp} \| + \| P_{V_1'} P_{V_2^\perp} - P_{V_1'} P_{V_2'^\perp} \| \\
		&\leq \| P_{V_1} - P_{V_1'} \| + \| P_{V_2} - P_{V_2'} \| \\
		&= \sin \angle (V_1,V_1') + \sin \angle(V_2,V_2').
\end{align*}
Theorem \ref{thm:subspace} can now be invoked to bound $\sin\angle(V_1,V_1')$ and $\sin\angle(V_2,V_2')$, and the claim follows.  
\end{proof}

Again, the factor of $\sqrt{l}$ appearing in \eqref{eq:subspace2:bnd} follows from the analogous factor appearing in \eqref{eq:subspace:bnd}.  Indeed, if this factor could be removed from \eqref{eq:subspace:bnd}, then the proof above shows that it would also be removed from \eqref{eq:subspace2:bnd}.  

For comparison, we present the following version of Theorem \ref{thm:modified:wedin}, which follows Theorem \ref{thm:modified:wedin3} and the argument above.  Again denote $\delta_0 :=\delta_1$.
\begin{theorem}[Modified Davis-Kahan-Wedin sine theorem: singular space]  \label{thm:modified:wedin2}
Suppose $A$ has rank $r$, and let $1 \leq j \leq l \leq r$ be integers. Then, for an arbitrary matrix $E$, 
$$ \sin \angle (V, V') \leq 4\frac{\|E\|}{\min\{\delta_{j-1}, \delta_l \}}, $$
where $V$ and $V'$ are defined in \eqref{def:subspaceinterval}.  
\end{theorem}


We now consider the problem of approximating the $j$-th singular vector $v_j$ recursively in terms of the bounds for $\sin \angle(v_i, v_i')$, $i < j$.  

\begin{theorem} \label{thm:general}
Assume that $E$ is $(C_1, c_1, \gamma)$-concentrated for a trio of constants $C_1, c_1, \gamma >0$.  Suppose $A$ has rank $r$, and let $1 \leq j \leq r$ be an integer.  Then, for any $t>0$,
$$ \sin \angle (v_j, v_j') \leq 4 \sqrt{2} \left( \left( \sum_{i=1}^{j-1} \sin^2 \angle (v_i, v_i') \right)^{1/2} + \frac{t r^{1/\gamma}}{\delta_j} + \frac{\|E\|^2}{\sigma_j \delta_j} + \frac{\|E\|}{\sigma_j} \right) $$
with probability at least 
$$ 1 - 6C_1 9^j \exp \left( -c_1 \frac{\delta_j^\gamma}{8^\gamma} \right) - 2C_1 9^{2r} \exp \left( -c_1 r \frac{t^\gamma}{4^\gamma} \right). $$ 
\end{theorem}

The bound in Theorem \ref{thm:general} depends inductively on the bounds for $\sin^2 \angle (v_i, v_i')$, $i = 1, \ldots, j-1$, and as such, we do not believe it to be sharp.  The bound does, however, improve upon a similar recursive bound presented in \cite{V}.

Finally, let us present the general form of Theorem \ref{thm:probweyl0} for singular values. Readers can compare the result with the classical bound in Theorem \ref{theorem:Weyl}.

\begin{theorem} \label{thm:probweyl}
Assume that $E$ is $(C_1, c_1, \gamma)$-concentrated for a trio of constants $C_1, c_1, \gamma >0$.  Suppose $A$ has rank $r$, and let $1 \leq j \leq r$ be an integer.  Then, for any $t>0$, 
\begin{equation} \label{eq:probweylbndlower}
	\sigma_j' \geq \sigma_j - t
\end{equation} 
with probability at least 
$$ 1 - 2C_1 9^j \exp \left( -c_1 \frac{t^\gamma}{4^\gamma} \right), $$
and
\begin{equation} \label{eq:probweylbndupper}
	\sigma_j' \leq \sigma_j + t r^{1/\gamma} + 2\sqrt{j} \frac{ \|E\|^2}{\sigma_j'} + j \frac{\|E\|^3}{{\sigma_j'}^2} 
\end{equation}
with probability at least 
$$ 1 - 2C_1 9^{2r} \exp \left( -c_1 r \frac{t^\gamma}{4^\gamma} \right). $$
\end{theorem}

\begin{remark}
Notice that the  upper bound for $\sigma_j'$ given in \eqref{eq:probweylbndupper} involves 
$1/\sigma_j'$.  In many situations, the lower bound in \eqref{eq:probweylbndlower} can be used to provide an upper bound for $1/\sigma_j'$.  
\end{remark}

We conjecture that the factors of $\sqrt{j}$ and $j$ appearing in \eqref{eq:probweylbndupper} are not needed and are simply an artifact of our proof.  In applications, $j$ is typically much smaller than the dimension, often making the contribution from these terms negligible.  To illustrate this point, consider the following example when $r = O(1)$.  Let $A$ and $E$ be symmetric matrices, and assume the entries on and above the diagonal of $E$ are independent random variables.  Such a matrix $E$ is known as a Wigner matrix, and the eigenvalues of perturbed Wigner matrices have been well-studied in the random matrix theory literature; see, for instance, \cite{KY, RS} and references therein.  In particular, the results in \cite{KY, RS} give the asymptotic location of the largest $r$ eigenvalues as well as their joint fluctuations.  These exact asymptotic results imply that, in this setting, the bounds appearing in Theorem \ref{thm:probweyl} are sharp, up to constant factors.  

As the bounds in Theorem \ref{thm:probweyl} are fairly general, let us state a corollary in the case when the entries of $E$ are independent random variables.  
\begin{corollary}  \label{cor:weyl}
Assume that $E$ is an $m \times n$ random matrix with independent entries which have mean zero and are bounded almost surely in magnitude by $K$ for some $K > 0$.  Suppose $A$ has rank $r$.  Then, for every $\eps > 0$, there exists $C_0 > 0$ (depending only on $\eps$ and $K$) such that, with probability at least $1 - \eps$,
\begin{equation} \label{eq:bnd:weyl}
	\sigma_j - C_0 \sqrt{j} \leq \sigma_j' \leq \sigma_j + C_0 \sqrt{r} + 2\sqrt{j} \frac{ \|E\|^2}{\sigma_j'} + j \frac{\|E\|^3}{{\sigma_j'}^2}  
\end{equation}
for all $1 \leq j \leq r$. 
\end{corollary}
Corollary \ref{cor:weyl} is an immediate consequence of Theorem \ref{thm:probweyl}, Lemma \ref{lemma:conc-nonsym}, and the union bound.  In particular, the bound in \eqref{eq:bnd:weyl} holds for all values of $1 \leq j \leq r$ simultaneously with probability at least $1 - \eps$. 

\subsection{Related results}
To conclude this section, let us mention a few related results.  
In \cite{V}, the second author managed to prove 
\begin{equation*}  
	\sin ^2 \angle(v_1, v'_1)  \le  C \frac{ \sqrt{r \log n} }{\delta } 
\end{equation*}  
under certain conditions.  While the right-hand side is quite close to the  optimal form 
 in Theorem \ref{theorem:main1}, the main problem here is that in  the left-hand side  one needs to square the sine function. 
 The bound for 
 $\sin \angle (v_i, v_i') $ with $i \ge 2$ was done by an inductive argument and was rather complicated.  Finally, the problem of estimating the singular values was not addressed at all in  \cite{V}.

Related results have also been obtained in the case where the random matrix $E$ contains Gaussian entries.  In \cite{RRW}, R.~Wang estimates the non-asymptotic distribution of the singular vectors when the entries of $E$ are iid standard normal random variables.  Recently, Allez and Bouchaud have studied the eigenvector dynamics of $A+E$ when $A$ is a real symmetric matrix and $E$ is a symmetric Brownian motion (that is, $E$ is a diffusive matrix process constructed from a family of independent real Brownian motions) \cite{AB}.  Our results also seems to have a close tie to 
the study of spiked covariance matrices, where a different kind of perturbation has been considered; see \cite{Ma, John,Nadler} for details. It would be  interesting to 
find a common generalization for these problems.

\section{Overview and outline}

We now briefly give an overview of the paper and discuss some of the key ideas behind the proof of our main results.  For simplicity, let us assume that $A$ and $E$ are $n \times n$ real symmetric matrices.  (In fact, we will symmetrize the problem in Section \ref{sec:prelim} below.)  Let $\sigma_1 \geq \cdots \geq \sigma_n$ be the eigenvalues of $A$ with corresponding (orthonormal) eigenvectors $v_1, \ldots, v_n$.  Let $\sigma_1'$ be the largest eigenvalue of $A+E$ with corresponding (unit) eigenvector $v_1'$.

Suppose we wish to bound $\sin \angle(v_1, v_1')$ (from Theorem \ref{thm:main}).  Since 
$$ \sin^2 \angle (v_1, v_1') = 1- \cos^2 \angle(v_1, v_1') = \sum_{k=2}^n | v_k \cdot v_1' |^2, $$
it suffices to bound $|v_k \cdot v_1'|$ for $k=2, \ldots, n$.  Let us consider the case when $k=2, \ldots, r$.  In this case, we have
$$ v_k^\mathrm{T} (A+E) v_1' - v_k^\mathrm{T} A v_1' = v_k^\mathrm{T} E v_1'. $$
Since $(A+E) v_1' = \sigma_1' v_1'$ and $v_k^\mathrm{T} A = \sigma_k v_k$, we obtain
$$ |\sigma_1' - \sigma_k| |v_k \cdot v_1'| \leq | v_k^\mathrm{T}E v_1' |. $$

Thus, the problem of bounding $|v_k \cdot v_1'|$ reduces to obtaining an upper bound for $| v_k^\mathrm{T}E v_1' |$ and a lower bound for the gap $|\sigma_1' - \sigma_k|$.  We will obtain bounds for both of these terms by using the concentration property (Definition \ref{def:concentration}).  

More generally, in Section \ref{sec:prelim}, we will apply the concentration property to obtain lower bounds for the gaps $\sigma_j' - \sigma_k$ when $j < k$, which will hold with high probability.  Let us illustrate this by now considering the gap $\sigma_1' - \sigma_2$.  Indeed, we note that
$$ \sigma_1' = \| A + E \| \geq v_1^\mathrm{T} (A+E) v_1 = \sigma_1 + v_1^\mathrm{T} E v_1. $$
Applying the concentration property \eqref{eq:concentration}, we see that $\sigma_1' > \sigma_1 - t$ with probability at least $1 - C_1 \exp(- c_1 t^{\gamma})$.  As $\delta := \sigma_1 - \sigma_2$, we in fact observe that
$$ \sigma_1' - \sigma_2 = \sigma_1' - \sigma_1 + \delta > \delta - t. $$
Thus, if $\delta$ is sufficiently large, we have (say) $\sigma_1' - \sigma_2 \geq \delta/2$ with high probability.  

In Section \ref{sec:proof}, we will again apply the concentration property to obtain upper bounds for terms of the form $v_k E v_j'$.  At the end of Section \ref{sec:proof}, we combine these bounds to complete the proof of Theorems  \ref{thm:main}, \ref{thm:subspace}, \ref{thm:general}, and \ref{thm:probweyl}.  In Section \ref{sec:concentration}, we discuss the $(C_1,c_1,\gamma)$-concentration property (Definition \ref{def:concentration}).  In particular, we generalize some previous results obtained by the second author in \cite{V}.  Finally, in Section \ref{section:app}, we present some applications of our main results.

Singular subspace perturbation bounds are applicable to a wide variety of problems.  For instance, \cite{CZ} discuss several applications of these bounds to high-dimensional statistics including high dimensional clustering, canonical correlation analysis (CCA), and matrix recovery.  In Section \ref{section:app}, we show how our results can be applied to the matrix recovery problem. The general matrix recovery problem  is the following.  $A$ is a large matrix.  However, the matrix $A$ is unknown to us.  We can only observe its noisy perturbation $A+E$, or in some cases just a small portion of the perturbation.  Our goal is to reconstruct $A$  or estimate an  important parameter
as accurately as possible from this observation. 
Furthermore,  several problems from combinatorics and theoretical computer science  can also be formulated in this setting. 
Special instances  of the matrix recovery problem have been investigated by many researchers using  spectral techniques  and combinatorial arguments in ingenious ways \cite{AM, AK, AKS,AzarMc,CCS,CP,CR,CRT,CT,Cest, DGP, KMO,KMO2,Krank,KLT,Kucera,MHT,Mc,NW,RVsamp}.  

We propose   the following simple analysis:   if  $A$ has rank $r$ and $1 \le j \le r$, then the projection of $A+E$ on the subspace  $V'$ spanned by the first $j$ singular vectors of $A+E$ is close to the projection of $A+E$ onto  the subspace  $V$ spanned by the first $j$ singular vectors of $A$, as our new results show that $V$ and $V'$ are very  close. Moreover, we can also show that the projection of $E$  onto $V$ is typically small.  Thus,  by projecting $A+E$ onto $V'$, we obtain  a good approximation of the rank $j$ approximation of $A$.   In certain cases, we  can repeat the above  operation a few times to obtain sufficient 
information to recover $A$ completely or to estimate the required parameter  with high accuracy and certainty.

\section{Preliminary tools} \label{sec:prelim}

In this section, we present some of the preliminary tools we will need to prove Theorems \ref{thm:main}, \ref{thm:subspace}, \ref{thm:general}, and \ref{thm:probweyl}.  

To begin, we define the $(m+n) \times (m+n)$ symmetric block matrices
\begin{equation} \label{eq:def:tildeA}
	\tilde{A} := \begin{bmatrix} 0 & A \\ A^\mathrm{T} & 0 \end{bmatrix} 
\end{equation}
and 
$$ \tilde{E} := \begin{bmatrix} 0 & E \\ E^\mathrm{T} & 0 \end{bmatrix}. $$

We will work with the matrices $\tilde{A}$ and $\tilde{E}$ instead of $A$ and $E$. If $A^Tu  = \sigma v$ and $Av=\sigma u$,  then $\tilde{A}^T (u^T, v^T)^T=\sigma (u^T, v^T)^T$ and $\tilde{A}^T (u^T, -v^T)^T= -\sigma (u^T, -v^T)^T$. In particular, the non-zero eigenvalues of $\tilde{A}$ are $\pm \sigma_1, \ldots, \pm \sigma_r$ and the eigenvectors are formed from the left and right singular vectors of $A$.  Similarly, the non-trivial eigenvalues of $\tilde{A} + \tilde{E}$ are $\pm \sigma_1', \ldots, \pm \sigma_{\min\{m,n\}}'$ (some of which may be zero) and the eigenvectors are formed from the left and right singular vectors of $A+E$.  

Along these lines, we introduce the following notation, which differs from the notation used above.  The non-zero eigenvalues of $\tilde{A}$ will be denoted by $\pm \sigma_1, \ldots, \pm \sigma_r$ with orthonormal eigenvectors $u_k$, $k=\pm 1, \ldots, \pm r$ such that
$$ \tilde{A} u_k = \sigma_k u_k, \qquad \tilde{A} u_{-k} = - \sigma_k u_{-k}, \qquad k = 1, \ldots, r. $$
Let $v_1, \ldots, v_j$ be the orthonormal eigenvectors of $\tilde{A}+\tilde{E}$ corresponding to the $j$-largest eigenvalues $\lambda_1 \geq \cdots \geq \lambda_j$.

In order to prove Theorems \ref{thm:main}, \ref{thm:subspace}, \ref{thm:general}, and \ref{thm:probweyl}, it suffices to work with the eigenvectors and eigenvalues of the matrices $\tilde{A}$ and $\tilde{A}+\tilde{E}$.  Indeed, Proposition \ref{prop:sine} will bound the angle between the singular vectors of $A$ and $A+E$ by the angle between the corresponding eigenvectors of $\tilde{A}$ and $\tilde{A} + \tilde{E}$.  

\begin{proposition} \label{prop:sine}
Let $u_1, v_1 \in \mathbb{R}^m$ and $u_2, v_2 \in \mathbb{R}^n$ be unit vectors.  Let $u, v \in \mathbb{R}^{m+n}$ be given by
$$ u = \begin{bmatrix} u_1 \\ u_2 \end{bmatrix}, \quad v = \begin{bmatrix} v_1 \\ v_2 \end{bmatrix}.  $$
Then
$$ \sin^2 \angle(u_1, v_1) + \sin^2 \angle(u_2, v_2) \leq 2 \sin^2 \angle(u, v). $$
\end{proposition}
\begin{proof}
Since $\|u\|^2 = \|v\|^2 = 2$, we have
\begin{align*}
	\cos^2 \angle(u,v) = \frac{1}{4} |u \cdot v|^2 \leq \frac{1}{2} |u_1 \cdot v_1|^2 + \frac{1}{2} |u_2 \cdot v_2|^2.
\end{align*}
Thus, 
\begin{align*}
	\sin^2 \angle(u,v) = 1 - \cos^2 \angle(u,v) \geq \frac{1}{2} \sin^2 \angle(u_1, v_1) + \frac{1}{2} \sin^2\angle(u_2, v_2),
\end{align*}
and the claim follows.
\end{proof}

We now introduce some useful lemmas.  The first lemma below, states that if $E$ is $(C_1,c_1,\gamma)$-concentrated, then $\tilde{E}$ is $(\tilde{C}_1,\tilde{c}_1,\gamma)$-concentrated, for some new constants $\tilde{C}_1 := 2C_1$ and $\tilde{c}_1:= c_1/2^{\gamma}$. 

\begin{lemma} \label{lemma:tilde}
Assume that $E$ is $(C_1, c_1,\gamma)$-concentrated for a trio of constants $C_1, c_1, \gamma >0$.  Let $\tilde{C}_1 := 2C_1$ and $\tilde{c}_1:= c_1/2^{\gamma}$.  Then for all unit vectors $u,v \in \mathbb{R}^{n+m}$, and every $t>0$,
\begin{equation} \label{eq:tilde-concentration}
	\Prob( |u^t \tilde{E} v| > t ) \leq \tilde{C}_1 \exp(-\tilde{c}_1 t^\gamma ).
\end{equation}	
\end{lemma}
\begin{proof}
Let 
$$ u = \begin{bmatrix} u_1 \\ u_2 \end{bmatrix}, \quad v = \begin{bmatrix} v_1 \\ v_2 \end{bmatrix} $$
be unit vectors in $\mathbb{R}^{m+n}$, where $u_1, v_1 \in \mathbb{R}^m$ and $u_2,v_2 \in \mathbb{R}^n$.  We note that
$$  u^{\mathrm{T}} \tilde{E} v  = u_1^\mathrm{T} E v_2 + u_2^\mathrm{T} E^\mathrm{T} v_1. $$
Thus, if any of the vectors $u_1, u_2, v_1, v_2$ are zero, \eqref{eq:tilde-concentration} follows immediately from \eqref{eq:concentration}.  Assume all the vectors $u_1, u_2, v_1, v_2$ are nonzero.  Then
$$ | u^{\mathrm{T}} \tilde{E} v | = |u_1^\mathrm{T} E v_2 + u_2^\mathrm{T} E^\mathrm{T} v_1| \leq \frac{|u_1^\mathrm{T} E v_2|}{\|u_1\| \|v_2\|} + \frac{|v_1^\mathrm{T} E u_2|}{\|u_2\|\|v_1\|}. $$
Thus, by \eqref{eq:concentration}, we have
\begin{align*}
	\Prob( |u^\mathrm{T} \tilde{E} v| > t) &\leq \Prob \left( \frac{|u_1^\mathrm{T} E v_2|}{\|u_1\| \|v_2\|}  > \frac{t}{2} \right) + \Prob \left( \frac{|v_1^\mathrm{T} E u_2|}{\|u_2\|\|v_1\|} > \frac{t}{2} \right) \\
	&\leq 2C_1 \exp \left( -c_1 \frac{ t^\gamma }{ 2^\gamma } \right),
\end{align*}
and the proof of the lemma is complete.  
\end{proof}

We will also consider the spectral norm of $\tilde{E}$.  Since $\tilde{E}$ is a symmetric matrix whose eigenvalues in absolute value are given by the singular values of $E$, it follows that
\begin{equation} \label{eq:normE}
	\| \tilde{E} \| = \|E \|. 
\end{equation}

We introduce $\varepsilon$-nets as a convenient way to discretize a compact set.  Let $\varepsilon > 0$.  A set $X$ is an $\varepsilon$-net of a set $Y$ if for any $y \in Y$, there exists $x \in X$ such that $\|x-y\| \leq \varepsilon$.  The following estimate for the maximum size of an $\varepsilon$-net of a sphere is well-known (see for instance \cite{RV}).

\begin{lemma} \label{lemma:net}
A unit sphere in $d$ dimensions admits an $\varepsilon$-net of size at most 
$$ \left(1+\frac{2}{\varepsilon} \right)^d.$$
\end{lemma}

Lemmas \ref{lemma:r-norm}, \ref{lemma:largest}, and \ref{lemma:j-largest} below are consequences of  the concentration property \eqref{eq:tilde-concentration}.  

\begin{lemma} \label{lemma:r-norm}
Assume that $E$ is $(C_1, c_1, \gamma)$-concentrated for a trio of constants $C_1, c_1, \gamma >0$.  Let $A$ be a $m \times n$ matrix with rank $r$.  Let $U$ be the $(m+n) \times 2r$ matrix whose columns are the vectors $u_1, \ldots, u_r, u_{-1}, \ldots, u_{-r}$.  Then, for any $t > 0$,
$$ \Prob \left( \|U^\mathrm{T} \tilde{E}U\| >  t r^{1/\gamma} \right) \leq \tilde{C}_1 9^{2r} \exp\left( - \tilde{c}_1 r \frac{t^\gamma}{2^\gamma} \right). $$
\end{lemma}
\begin{proof}
Clearly $U^\mathrm{T} \tilde{E} U$ is a symmetric  $2r \times 2r$ matrix.  Let $S$ be the unit sphere in $\mathbb{R}^{2r}$.  Let $\mathcal{N}$ be a $1/4$-net of $S$.  It is easy to verify (see for instance \cite{RV}) that for any $2r \times 2r$ symmetric matrix $B$,
$$ \|B\| \leq 2 \max_{x \in \mathcal{N}} |x^\ast B x| . $$

For any fixed $x \in \mathcal{N}$, we have
$$ \Prob( |x^\mathrm{T} U^\mathrm{T} \tilde{E} U x | > t) \leq \tilde{C}_1 \exp(-\tilde{c}_1 t^\gamma) $$
by Lemma \ref{lemma:tilde}.  Since $|\mathcal{N}| \leq 9^{2r}$, we obtain
\begin{align*}
	\Prob ( \| U^\mathrm{T} \tilde{E} U \| > t r^{1/\gamma}) &\leq \sum_{x \in \mathcal{N}} \Prob\left( |x^\mathrm{T} U^\mathrm{T} \tilde{E} U x | > \frac{1}{2} t r^{1/\gamma} \right) \\
	& \leq \tilde{C}_1 9^{2r} \exp\left(-\tilde{c}_1 r \frac{t^\gamma}{2^{\gamma}} \right).
\end{align*}
\end{proof}

\begin{lemma} \label{lemma:largest}
Assume that $E$ is $(C_1, c_1, \gamma)$-concentrated for a trio of constants $C_1, c_1, \gamma >0$.  Suppose $A$ has rank $r$.  Then, for any $t > 0$, 
\begin{equation} \label{eq:lambda1bnd}
	\lambda_1 \geq \sigma_1 - t 
\end{equation}
with probability at least $1 - \tilde{C}_1 \exp(-\tilde{c}_1 t^{\gamma})$.  

In particular, if $\sigma_1 > 0$, then $\lambda_1 \geq \frac{\sigma_1}{2}$ with probability at least $1 - \tilde{C}_1 \exp \left( -\tilde{c}_1 \frac{\sigma_1^\gamma}{2^\gamma} \right)$.  If, in addition, $\delta > 0$, then 
$$ \lambda_1 - \sigma_k \geq \frac{1}{2} \delta $$
for $k = 2,\ldots, r $ with probability at least $1 - \tilde{C}_1 \exp \left( -\tilde{c}_1 \frac{\delta^\gamma}{2^\gamma} \right)$.
\end{lemma}
\begin{proof}
We observe that
$$ \lambda_1 = \|\tilde{A} + \tilde{E}\| \geq u_1^\mathrm{T} (\tilde{A} + \tilde{E}) u_1 = \sigma_1 + u_1^\mathrm{T} \tilde{E} u_1. $$
By Lemma \ref{lemma:tilde}, we have 
$$ \Prob( |u_1^\mathrm{T} \tilde{E} u_1| > t) \leq \tilde{C}_1 \exp( -\tilde{c}_1 t^\gamma) $$
for every $t > 0$, and \eqref{eq:lambda1bnd} follows.  

If $\sigma_1 > 0$, then the bound $\lambda_1 \geq \frac{\sigma_1}{2}$ can be obtained by taking $t = \sigma_1/2$ in \eqref{eq:lambda1bnd}.  Assume $\delta > 0$.  Taking $t=\delta/2$ in \eqref{eq:lambda1bnd} yields 
$$ \lambda_1 - \sigma_k \geq \lambda_1 - \sigma_2 = \lambda_1 - \sigma_1 + \delta \geq \frac{\delta}{2} $$
for $k=2, \ldots, r$ with probability at least $1 - \tilde{C}_1 \exp \left( -\tilde{c}_1 \frac{\delta^\gamma}{2^\gamma} \right)$.
\end{proof}

Using the Courant minimax principle, Lemma \ref{lemma:largest} can be generalized to the following.

\begin{lemma} \label{lemma:j-largest}
Assume that $E$ is $(C_1, c_1, \gamma)$-concentrated for a trio of constants $C_1, c_1, \gamma >0$.  Suppose $A$ has rank $r$, and let $1 \leq j \leq r$ be an integer.  Then, for any $t > 0$, 
\begin{equation} \label{eq:lambdajbnd1}
	\lambda_j \geq \sigma_j - t
\end{equation}
with probability at least $1 - \tilde{C}_1 9^j \exp\left( - \tilde{c}_1 \frac{t^\gamma}{2^\gamma} \right)$.  

In particular, $\lambda_j \geq \frac{\sigma_j}{2}$ with probability at least $1 - \tilde{C}_1 9^j \exp \left( - \tilde{c}_1 \frac{\sigma_j^\gamma}{4^\gamma} \right)$. In addition, if $\delta_j > 0$, then 
\begin{equation} \label{eq:lambdajbnd2}
	\lambda_j - \sigma_k \geq \frac{\delta_j}{2}
\end{equation}
for $k=j+1, \ldots, r$ with probability at least $1 - \tilde{C}_1 9^j \exp \left( - \tilde{c}_1 \frac{\delta_j^\gamma}{4^\gamma} \right)$. 
\end{lemma}
\begin{proof}
It suffices to prove \eqref{eq:lambdajbnd1}.  Indeed, the bound $\lambda_j \geq \frac{\sigma_j}{2}$ follows from \eqref{eq:lambdajbnd1} by taking $t = \sigma_j/2$, and \eqref{eq:lambdajbnd2} follows by taking $t = \delta_j/2$.

Let $S$ be the unit sphere in $\mathrm{Span}\{u_1,\ldots,u_j\}$.  By the Courant minimax principle,
\begin{align*}
	\lambda_j &= \max_{\dim(V)=j} \min_{\|v\|=1;v\in V} v^\mathrm{T} (\tilde{A}+\tilde{E})v \\
		& \geq \min_{v \in S} v^\mathrm{T} (\tilde{A}+\tilde{E})v \\
		& \geq \sigma_j + \min_{v \in S} v^\mathrm{T} \tilde{E} v.
\end{align*}
Thus, it suffices to show
$$ \Prob\left( \sup_{v \in S} |v^\mathrm{T} \tilde{E} v| > t \right) \leq \tilde{C}_1 9^j \exp\left( -\tilde{c}_1 \frac{t^\gamma}{2^\gamma} \right) $$
for all $t > 0$.  

Let $\mathcal{N}$ be a $1/4$-net of $S$.  By Lemma \ref{lemma:net}, $|\mathcal{N}| \leq 9^{j}$.  We now claim that 
\begin{equation} \label{eq:supmaxnet}
	T := \sup_{v \in S} | v^\mathrm{T} \tilde{E} v| \leq 2 \max_{ u \in \mathcal{N}} |u^\mathrm{T} \tilde{E} u|. 
\end{equation}
Indeed, fix a realization of $\tilde{E}$.  Since $S$ is compact, there exists $v \in S$ such that $T = |v^\mathrm{T} \tilde{E} v|$.  Moreover, there exists $x \in \mathcal{N}$ such that $\|x - v\| \leq 1/4$.  Clearly the claim is true when $x = v$; assume $x \neq v$.  Then, by the triangle inequality, we have
\begin{align*}
	T &\leq |v^\mathrm{T} \tilde{E} v - v^\mathrm{T} \tilde{E} x| + |v^\mathrm{T} \tilde{E} x - x^\mathrm{T} \tilde{E} x| + |x^\mathrm{T} \tilde{E} x| \\
		&\leq \frac{1}{4} \frac{ |v^\mathrm{T} \tilde{E} (v-x)| }{\| v-x\|} + \frac{1}{4} \frac{ |(v-x)^\mathrm{T} \tilde{E} x |}{\|v-x\| } + \sup_{u \in \mathcal{N}} |u^\mathrm{T} \tilde{E} u| \\
		& \leq \frac{T}{2} + \sup_{u \in \mathcal{N}} |u^\mathrm{T} \tilde{E} u|,
\end{align*}
and \eqref{eq:supmaxnet} follows.  

Applying \eqref{eq:supmaxnet} and Lemma \ref{lemma:tilde}, we have
\begin{align*}
	\Prob \left( \sup_{v \in S} |v^\mathrm{T} \tilde{E} v| > t \right) &\leq \sum_{u \in \mathcal{N}} \Prob \left( |u^\mathrm{T} \tilde{E} u| > \frac{t}{2} \right) \leq 9^j \tilde{C}_1 \exp \left( -\tilde{c}_1 \frac{t^\gamma}{2^\gamma} \right),
\end{align*}
and the proof of the lemma is complete.  
\end{proof}

We will continually make use of the following simple fact:
\begin{equation} \label{eq:aea}
	(\tilde{A} + \tilde{E}) - \tilde{A} = \tilde{E}.
\end{equation}

\section{Proof of Theorems \ref{thm:main}, \ref{thm:subspace}, \ref{thm:general}, and \ref{thm:probweyl}} \label{sec:proof}

This section is devoted to Theorems \ref{thm:main}, \ref{thm:subspace}, \ref{thm:general}, and \ref{thm:probweyl}.  To begin, define the subspace
$$ W:= \mathrm{Span}\{u_1, \ldots, u_r, u_{-1}, \ldots, u_{-r} \}.$$
Let $P$ be the orthogonal projection onto $W^\perp$.  

\begin{lemma} \label{lemma:proj_bound}
Assume that $E$ is $(C_1, c_1, \gamma)$-concentrated for a trio of constants $C_1, c_1, \gamma >0$.  Suppose $A$ has rank $r$, and let $1 \leq j \leq r$ be an integer.  Then 
\begin{equation} \label{eq:suppvi}
	\sup_{1 \leq i \leq j} \|P v_i \| \leq 2 \frac{\|E\|}{\sigma_j}
\end{equation}
with probability at least $1 - \tilde{C}_1 9^j \exp \left( - \tilde{c}_1 \frac{\sigma_j^\gamma}{4^\gamma} \right)$.  
\end{lemma}
\begin{proof}
Consider the event 
$$ \Omega_j := \left\{ \lambda_j \geq \frac{1}{2} \sigma_j \right\}. $$
By Lemma \ref{lemma:j-largest} (or Lemma \ref{lemma:largest} in the case $j=1$), $\Omega_j$ holds with probability at least $1 - \tilde{C}_1 9^j \exp \left( - \tilde{c}_1 \frac{\sigma_j^\gamma}{4^\gamma} \right)$.    

Fix $1 \leq i \leq j$.  By multiplying \eqref{eq:aea} on the left by $(P v_i)^\mathrm{T}$ and on the right by $v_i$, we obtain
$$ | \lambda_i (P v_i)^\mathrm{T} v_i | \leq \| P v_i\| \|\tilde{E}\| $$
since $(P v_i)^\mathrm{T} \tilde{A} = 0$.  Thus, on the event $\Omega_j$, we have
$$ \| P v_i \|^2 = |(P v_i)^\mathrm{T} v_i | \leq \frac{1}{\lambda_j}\|P v_i\| \|\tilde{E}\| \leq \frac{2}{\sigma_j} \| P v_i \|\|\tilde{E}\|. $$
We conclude that, on the event $\Omega_j$, 
$$ \sup_{1 \leq i \leq j} \|P v_i \| \leq 2 \frac{\|E\|}{\sigma_j}, $$
and the proof is complete.  
\end{proof}

\begin{lemma} \label{lemma:uproj}
Assume that $E$ is $(C_1, c_1, \gamma)$-concentrated for a trio of constants $C_1, c_1, \gamma >0$.  Suppose $A$ has rank $r$, and let $1 \leq j \leq r$ be an integer.  Define $U_j$ to be the $(m+n) \times (2r-j)$ matrix with columns $u_{j+1}, \ldots, u_r, u_{-1}, \ldots, u_{-r}$.  Then, for any $t>0$,
\begin{equation} \label{eq:sujtv}
	\sup_{1 \leq i \leq j} \| U_j^\mathrm{T} v_i \| \leq 4 \left( \frac{t r^{1/\gamma}}{\delta_j} + \frac{\|E\|^2}{\delta_j \sigma_j} \right)
\end{equation}
with probability at least 
$$ 1 - 2\tilde{C}_1 9^j \exp \left( - \tilde{c}_1 \frac{\delta_j^\gamma}{4^\gamma}\right) - \tilde{C}_1 9^{2r} \exp \left( - \tilde{c}_1r \frac{t^\gamma}{2^\gamma} \right). $$
\end{lemma}
\begin{proof}
Define the event
\begin{align*}
\Omega_j &:= \left\{ \sup_{1 \leq i \leq j} \|P v_i \| \leq 2 \frac{\|E\|}{\sigma_j} \right\} \bigcap \left\{ \| U^\mathrm{T} \tilde{E} U \| \leq t r^{1/\gamma} \right\} \bigcap \left\{  \lambda_j - \sigma_{j+1} \geq \frac{\delta_j}{2} \right\}. 
\end{align*}
By Lemmas \ref{lemma:r-norm}, \ref{lemma:j-largest}, and \ref{lemma:proj_bound}, it follows that
$$ \Prob( \Omega_j ) \geq 1 - 2\tilde{C}_1 9^j \exp \left( - \tilde{c}_1 \frac{\delta_j^\gamma}{4^\gamma}\right) - \tilde{C}_1 9^{2r} \exp \left( - \tilde{c}_1r \frac{t^\gamma}{2^\gamma} \right). $$

Fix $1 \leq i \leq j$.  We multiply \eqref{eq:aea} on the left by $U_j^\mathrm{T}$ and on the right by $v_i$ to obtain
\begin{equation} \label{eq:ujae}
	U_j^\mathrm{T} (\tilde{A} + \tilde{E}) v_i - U_j^\mathrm{T} \tilde{A} v_i = U_j^\mathrm{T} \tilde{E} v_i. 
\end{equation}
We note that
$$ U_j^\mathrm{T} (\tilde{A} + \tilde{E}) v_i = \lambda_i U_j^\mathrm{T} v_i $$
and
$$ U_j^\mathrm{T} \tilde{A} v_i = D_j U_j^\mathrm{T} v_i, $$
where $D_j$ is the diagonal matrix with the values $\sigma_{j+1}, \ldots, \sigma_r, -\sigma_{1}, \ldots, -\sigma_{r}$ on the diagonal.  

For the right-hand side of \eqref{eq:ujae}, we write $v_i = U U^\mathrm{T} v_i + P v_i$, where $U$ is the matrix with columns $u_1, \ldots, u_r, u_{-1}, \ldots, u_{-r}$ and $P$ is the orthogonal projection onto $W^\perp$.  Thus, on the event $\Omega_j$, we have
\begin{align*}
	\|U_j^\mathrm{T} \tilde{E} v_i\| \leq \|U_j^\mathrm{T} \tilde{E} U\| + \|\tilde{E}\| \|P v_i\| \leq t r^{1/\gamma} + 2 \frac{\|E\|^2}{\sigma_j}.
\end{align*}
Here we used the fact that $U_j^\mathrm{T} \tilde{E} U$ is a sub-matrix of $U^\mathrm{T} \tilde{E} U$ and hence 
$$ \|U_j^\mathrm{T} \tilde{E} U\| \leq \| U^\mathrm{T} \tilde{E} U\|. $$

Combining the above computations and bound yields
$$ \| (\lambda_i I - D_j) U_j^\mathrm{T} v_i \| \leq 2 \left( t r^{1/\gamma} + \frac{\|E\|^2}{\sigma_j} \right) $$
on the event $\Omega_j$.    

We now consider the entries of the diagonal matrix $\lambda_i I - D_j$.  On $\Omega_j$, we have that, for any $k \geq j+1$,
$$ \lambda_i - \sigma_k \geq \lambda_j - \sigma_{j+1} \geq \frac{\delta_j}{2}. $$
By writing the elements of the vector $U_j^\mathrm{T} v_i$ in component form, it follows that
$$ \|(\lambda_i I - D_j) U_j^\mathrm{T} v_i \| \geq \frac{\delta_j}{2} \| U_j^\mathrm{T} v_i \| $$
and hence
$$  \| U_j^\mathrm{T} v_i \| \leq 4 \left( \frac{t r^{1/\gamma}}{\delta_j} + \frac{\|E\|^2}{\sigma_j \delta_j} \right) $$
on the event $\Omega_j$.  Since this holds for each $1 \leq i \leq j$, the proof is complete.  
\end{proof}

With Lemmas \ref{lemma:proj_bound} and \ref{lemma:uproj} in hand, we now prove Theorems \ref{thm:main}, \ref{thm:subspace}, \ref{thm:general}, and \ref{thm:probweyl}.  By Proposition \ref{prop:sine}, in order to prove Theorems \ref{thm:main} and \ref{thm:general}, it suffices to bound $\sin \angle (u_j, v_j)$  because $u_j, v_j$ are formed from the left and right singular vectors of $A$ and $A+E$.  

\begin{proof}[Proof of Theorem \ref{thm:main}]
We write 
$$ v_1 = \sum_{k=1}^r \alpha_k u_k + \sum_{k=1}^r \alpha_{-k} u_{-k} + P v_1, $$  
where $P$ is the orthogonal projection onto $W^\perp$.  Then
$$ \sin^2 \angle (u_1, v_1) = 1- \cos^2 \angle (u_1, v_1) = \sum_{k=2}^r |\alpha_k|^2 + \sum_{k=1}^r |\alpha_{-k}|^2 + \| P v_1 \|^2. $$
Applying the bounds obtained from Lemmas \ref{lemma:proj_bound} and \ref{lemma:uproj} (with $j=1$), we obtain
$$ \sin^2 \angle (u_1, v_1) \leq 16 \left( \frac{t r^{1/\gamma}}{\delta} + \frac{\|E\|^2}{\sigma_1 \delta} \right)^2 + 4 \frac{\|E\|^2}{ \sigma_1^2} $$
with probability at least 
\begin{equation} \label{eq:probholdmain}
	1 - 27 \tilde{C}_1 \exp \left( -\tilde{c}_1 \frac{\delta^\gamma}{4^\gamma} \right) - \tilde{C}_1 9^{2r} \exp \left(- \tilde{c}_1 r \frac{t^\gamma}{2^\gamma} \right). 
\end{equation}
We now note that
\begin{align*}
	16 \left( \frac{t r^{1/\gamma}}{\delta} + \frac{\|E\|^2}{\sigma_1 \delta} \right)^2 + 4 \frac{\|E\|^2}{ \sigma_1^2}  &\leq 16  \left( \frac{t r^{1/\gamma}}{\delta} + \frac{\|E\|^2}{\sigma_1 \delta} + \frac{\|E\|}{ \sigma_1} \right)^2.
\end{align*}
The correct absolute constant in front can now be deduced from the bound above and Proposition \ref{prop:sine}.  The lower bound on the probability given in \eqref{eq:probholdmain} can be written in terms of the constants $C_1, c_1, \gamma$ by recalling the definitions of $\tilde{C}_1$ and $\tilde{c}_1$ given in Lemma \ref{lemma:tilde}.   
\end{proof}

\begin{proof}[Proof of Theorem \ref{thm:general}]
We again write 
\begin{equation} \label{eq:vjproj}
	v_j = \sum_{k=1}^r \alpha_k u_k + \sum_{k=1}^r \alpha_{-k} u_{-k} + P v_j, 
\end{equation}
where $P$ is the orthogonal projection onto $W^\perp$.  Then we have that
\begin{align*}
	\sin^2 \angle (u_j, v_j) &= 1- \cos^2 \angle (u_j, v_j) \\
		&= \sum_{k=1}^{j-1} |\alpha_k|^2 + \sum_{k=j+1}^r |\alpha_k|^2 + \sum_{k=1}^r |\alpha_{-k}|^2 + \|P v_j\|^2. 
\end{align*}
For any $1 \leq k \leq j-1$, we have that
$$ |\alpha_k|^2 = | v_j \cdot (u_k - v_k) |^2 \leq \|v_k - u_k\|^2 \leq 2 (1 - \cos \angle (v_k,u_k)) \leq 2 \sin^2 \angle(v_k,u_k). $$
Moreover, from Lemmas \ref{lemma:proj_bound} and \ref{lemma:uproj}, we have
$$ \sum_{k=j+1}^r |\alpha_k|^2 + \sum_{k=1}^r |\alpha_{-k}|^2 \leq 16 \left( \frac{t r^{1/\gamma}}{\delta_j} + \frac{\|E\|^2}{\sigma_j \delta_j} \right)^2 $$
with probability at least 
$$ 1 - 2\tilde{C}_1 9^j \exp \left( - \tilde{c}_1 \frac{\delta_j^\gamma}{4^\gamma}\right) - \tilde{C}_1 9^{2r} \exp \left( - \tilde{c}_1r \frac{t^\gamma}{2^\gamma} \right). $$ 
and
$$ \| P v_j \|^2 \leq 4  \frac{\|E\|^2}{\sigma_j^2} $$
with probability at least $1 - \tilde{C}_1 9^j \exp \left( - \tilde{c}_1 \frac{\sigma_j^\gamma}{4^\gamma} \right)$.  The proof of Theorem \ref{thm:general} is complete by combining the bounds above\footnote{Here the bounds are given in terms of $\sin^2 \angle(v_k, u_k)$ for $1 \leq k \leq j-1$.  However, $u_k$ and $v_k$ are formed from the left and right singular vectors of $A$ and $A+E$.  To avoid the dependence on both the left and right singular vectors, one can begin with \eqref{eq:vjproj} and consider only the coordinates of $v_j$ which correspond to the left (alternatively right) singular vectors.  By then following the proof for only these coordinates, one can bound the left (right) singular vectors by terms which only depend on the previous left (right) singular vectors.}.  As in the proof of Theorem \ref{thm:main}, the correct constant factor in front can be deduced from Proposition \ref{prop:sine}.  
\end{proof}

\begin{proof}[Proof of Theorem \ref{thm:subspace}]
Define the subspaces
$$ \tilde{U}:= \mathrm{Span}\{u_1, \ldots, u_j \} \quad \text{and} \quad \tilde{V}:= \mathrm{Span}\{v_1, \ldots, v_j\}. $$
By Proposition \ref{prop:sine}, it suffices to bound $\sin \angle(\tilde{U}, \tilde{V})$.  

Let $Q$ be the orthogonal projection onto $\tilde{U}^\perp$.  By Lemmas \ref{lemma:proj_bound} and \ref{lemma:uproj}, it follows that
\begin{equation} \label{eq:supqvi}
	\sup_{1 \leq i \leq j} \|Q v_i \| \leq 4 \left( \frac{t r^{1/\gamma} }{\delta_j} + \frac{\|E\|^2}{\sigma_j \delta_j} + \frac{\|E\|}{\sigma_j} \right) 
\end{equation}
with probability at least 
$$ 1 - 3\tilde{C}_1 9^j \exp \left( - \tilde{c}_1 \frac{\delta_j^\gamma}{4^\gamma}\right) - \tilde{C}_1 9^{2r} \exp \left( - \tilde{c}_1r \frac{t^\gamma}{2^\gamma} \right). $$
On the event where \eqref{eq:supqvi} holds, we have
$$ \sup_{v \in \tilde{V}, \|v\| = 1} \| Q v \| \leq 4 \sqrt{j} \left( \frac{t r^{1/\gamma} }{\delta_j} + \frac{\|E\|^2}{\sigma_j \delta_j} + \frac{\|E\|}{\sigma_j} \right)  $$
by the triangle inequality and the Cauchy-Schwarz inequality.  Thus, by \eqref{eq:ssad}, we conclude that
\begin{align*}
	\sin \angle(\tilde{U}, \tilde{V}) &\leq 4 \sqrt{j} \left( \frac{t r^{1/\gamma} }{\delta_j} + \frac{\|E\|^2}{\sigma_j \delta_j} + \frac{\|E\|}{\sigma_j} \right)
\end{align*}
on the event where \eqref{eq:supqvi} holds.  The claim now follows from Proposition \ref{prop:sine}.  
\end{proof}

\begin{proof}[Proof of Theorem \ref{thm:probweyl}]
The lower bound \eqref{eq:probweylbndlower} follows from Lemma \ref{lemma:j-largest}; it remains to prove \eqref{eq:probweylbndupper}.  Let $U$ be the $(m+n) \times 2r$ matrix whose columns are given by the vectors $u_1, \ldots, u_r, u_{-1}, \ldots, u_{-r}$, and recall that $P$ is the orthogonal projection onto $W^\perp$.  

Let $S$ denote the unit sphere in $\mathrm{Span}\{v_1, \ldots, v_j\}$.  Then for $1 \leq i \leq j$, we multiply \eqref{eq:aea} on the left by $v_i^\mathrm{T} P$ and on the right by $v_i$ to obtain
$$ \lambda_i \|P v_i \|^2 \leq \| v_i^\mathrm{T} P \tilde{E} v_i \| \leq \|P v_i \| \|E\|. $$
Here we used \eqref{eq:normE} and the fact that $P \tilde{A} = 0$.  Therefore, we have the deterministic bound
$$ \sup_{1 \leq i \leq j} \| P v_i \| \leq \frac{ \| E\|}{\lambda_j}. $$
By the Cauchy-Schwarz inequality, it follows that
\begin{equation} \label{eq:detsupbnd}
	\sup_{v \in S} \| P v \| \leq \sqrt{j} \frac{ \|E \| }{\lambda_j}. 
\end{equation}

By the Courant minimax principle, we have
\begin{align*}
	\sigma_j = \max_{\dim(V)=j} \min_{v\in V,\|v\|=1} v^\mathrm{T} \tilde{A}v \geq \min_{v \in S} v^\mathrm{T} \tilde{A} v \geq \lambda_j - \max_{v \in S} |v^\mathrm{T} \tilde{E} v|. 
\end{align*}
Thus, it suffices to show that
\begin{equation*} 
	\max_{v \in S} |v^\mathrm{T} \tilde{E} v| \leq t r^{1/\gamma} + 2\sqrt{j} \frac{ \|E\|^2}{\lambda_j} + j \frac{\|E\|^3}{{\lambda_j}^2} 
\end{equation*}
with probability at least $1-\tilde{C}_1 9^{2r} \exp \left( - \tilde{c}_1 r \frac{t^\gamma}{2^\gamma} \right)$.  

We decompose $v = Pv + U U^\mathrm{T} v$ and obtain
\begin{align*}
	 \max_{v \in S} |v^\mathrm{T} \tilde{E} v| \leq \max_{v \in S} \|Pv\|^2 \|\tilde{E}\| + 2 \max_{v \in S} \| Pv \| \|\tilde{E}\| + \| U^\mathrm{T} \tilde{E} U \|.
\end{align*}
Thus, by Lemma \ref{lemma:r-norm} and \eqref{eq:detsupbnd}, we have
$$ \max_{v \in S} |v^\mathrm{T} \tilde{E} v| \leq j \frac{ \|E\|^3}{\lambda_j^2} + 2 \sqrt{j} \frac{ \|E\|^2}{\lambda_j} + t r^{1/\gamma} $$
with probability at least $1-\tilde{C}_1 9^{2r} \exp \left( - \tilde{c}_1 r \frac{t^\gamma}{2^\gamma} \right)$, and the proof is complete.   
\end{proof}

\section{The concentration property} \label{sec:concentration}

In this section, we give examples of random matrix models satisfying  Definition \ref{def:concentration}.  

\begin{lemma} \label{lemma:bernoulli}
There exists a constant $C_1$ such that the following holds.  Let $E$ be a random $n \times n$ Bernoulli matrix.  Then 
$$ \Prob ( \|E\| > 3 \sqrt{n} ) \leq \exp(-C_1 n), $$
and for any fixed unit vectors $u,v$ and positive number $t$,
$$ \Prob (|u^\mathrm{T} E v| \geq t ) \leq 2 \exp(-t^2/2). $$
\end{lemma}

The bounds in Lemma \ref{lemma:bernoulli} also hold for the case where the noise is Gaussian (instead of Bernoulli).  Indeed, when the entries of $E$ are iid standard normal random variables, $u^{\mathrm{T}} E v$ has the standard normal distribution. The first bound is a corollary of a general concentration result from \cite{V}. It can also be proved directly using a net argument. The second bound follows from Azuma's inequality \cite{Azuma, H, Ste}; see also \cite{V} for a direct proof with a more generous constant.

We now verify the $(C_1,c_1,\gamma)$-concentration property for slightly more general random matrix models.  We will discuss these matrix models further in Section \ref{section:app}.  In the lemmas below, we consider both the case where $E$ is a real symmetric random matrix with independent entries and when $E$ is a non-symmetric random matrix with independent entries.  

\begin{lemma} \label{lemma:conc-sym}
Let $E = (\xi_{ij})_{i,j=1}^n$ be a $n \times n$ real symmetric random matrix where 
$$ \{ \xi_{ij} : 1 \leq i \leq j \leq n\} $$ 
is a collection of independent random variables each with mean zero.  Further assume
$$ \sup_{1 \leq i \leq j \leq n} |\xi_{ij}| \leq K $$  
with probability $1$, for some $K \geq 1$.  Then for any fixed unit vectors $u,v$ and every $t > 0$
$$ \Prob (|u^\mathrm{T} E v| \geq t) \leq 2 \exp \left( \frac{-t^2}{8 K^2 } \right). $$
\end{lemma}
\begin{proof}
We write
$$ u^\mathrm{T} E v = \sum_{1 \leq i < j \leq n} (u_i v_j + v_i u_j) \xi_{ij} + \sum_{i=1}^n u_i v_i \xi_{ii}. $$
As the right side is a sum of independent, bounded random variables, we apply Hoeffding's inequality (\cite[Theorem 2]{H}) to obtain 
$$ \Prob (|u^\mathrm{T} E v - \E u^\mathrm{T} E v| \geq t) \leq 2 \exp \left( \frac{-t^2}{8 K^2 } \right). $$
Here we used the fact that
$$ \sum_{1 \leq i < j \leq n} (|u_i| |v_j| + |v_i| |u_j|)^2 + \sum_{i=1}^n |u_i|^2 |v_i|^2 \leq 4 \sum_{i,j=1}^n |u_i|^2 |v_j|^2 \leq 4 $$
because $u,v$ are unit vectors.  Since each $\xi_{ij}$ has mean zero, it follows that $\E u^\mathrm{T} E v = 0$, and the proof is complete.  
\end{proof}

\begin{lemma} \label{lemma:conc-nonsym}
Let $E = (\xi_{ij})_{1 \leq i \leq m, 1 \leq j \leq n}$ be a $m \times n$ real random matrix where
$$ \{ \xi_{ij} : 1 \leq i \leq m, 1 \leq j \leq n \} $$
is a collection of independent random variables each with mean zero.  Further assume
$$ \sup_{1 \leq i \leq m, 1 \leq j \leq n} |\xi_{ij}| \leq K $$
with probability $1$, for some $K \geq 1$.  Then for any fixed unit vectors $u \in \mathbb{R}^m, v \in \mathbb{R}^n$, and every $t > 0$
\begin{equation} \label{eq:concprop-nonsym}
	\Prob( |u^\mathrm{T} E v| \geq t) \leq 2 \exp \left( \frac{-t^2}{2 K^2} \right). 
\end{equation}
\end{lemma}

The proof of Lemma \ref{lemma:conc-nonsym} is nearly identical to the proof of lemma \ref{lemma:conc-sym}.  Indeed, \eqref{eq:concprop-nonsym} follows from Hoeffding's inequality since $u^\mathrm{T}E v$ can be the written as the sum of independent random variables; we omit the details.  

Many other models of random matrices satisfy Definition \ref{def:concentration}.  If the entries of $E$ are independent and have a rapidly decaying tail, then $E$ will be $(C_1,c_1,\gamma)$-concentrated for some constants $C_1,c_1,\gamma>0$.  One can achieve this by standard truncation arguments.  
For many arguments of this type, see for instance \cite{VW}.  As an example, we present a concentration result from \cite{RV} when the entries of $E$ are iid sub-exponential random variables.

\begin{lemma}[Proposition 5.16 of \cite{RV}] \label{lemma:conc-sub}
Let $E = (\xi_{ij})_{1 \leq i \leq m, 1 \leq j \leq n}$ be a $m \times n$ real random matrix whose entries $\xi_{ij}  $ are iid copies of a sub-exponential random variable $\xi$ with constant $K$, i.e. $\Prob(|\xi| > t) \le \exp(1-t/K)$ for all $t>0$.  Assume $\xi$ has mean 0 and variance 1. Then there are constants $C_1, c_1>0$ (depending only on $K$) such that for any fixed unit vectors $u \in \mathbb{R}^m, v \in \mathbb{R}^n$ and any $t > 0$, one has
\begin{equation*}
	\Prob( |u^\mathrm{T} E v| \geq t) \leq C_1 \exp \left( -c_1 t \right). 
\end{equation*}

\end{lemma}

Finally, let us point out that  the assumption that the entries are independent is not necessary. As an example, we mention random orthogonal matrices. 
For another example, one can consider the elliptic ensembles;  this can be verified using standard truncation and concentration results, see for instance \cite{KS, LT, McDiarmid, RV} and \cite[Chapter 5]{BS}.

\section{An application: The matrix recovery problem}  \label{section:app}

 The matrix recovery problem  is the following: $A$ is a large unknown matrix.  We can only observe its noisy image  $A+E$, or in some cases just a small part of it.  We would like to reconstruct $A$  or estimate an important parameter as accurately as possible from this observation.  

 Consider a deterministic $m \times n$ matrix $$A = (a_{ij})_{1 \leq i \leq m, 1 \leq j \leq n.}$$  
Let $Z$ be a random matrix of the same size whose entries
 $\{z_{ij} : 1 \leq i \leq m, 1 \leq j \leq n \}$  are independent random variables with mean zero and unit variance.   For convenience, we will assume that $\| Z \| _{\infty} := \max_{i,j} |z_{ij}| \le K$, for some fixed $K >0$, with probability $1$.

Suppose  that  we have only partial access to the noisy data $A+Z$. Each entry of this matrix is observed with probability $p$ and 
 unobserved with probability $1-p$ for some small $p$.  We will write $0$ if the entry is not observed.  Given this sparse observable data matrix $B$, the  task is to reconstruct $A$. 
 
 The matrix completion problem is a central one in data analysis, and there is a large collection of literature focusing on the low rank case; see  \cite{AM,CCS,CP,CR,CRT,CT,Cest,KMO,KMO2,Krank,KLT,MHT,NW,RVsamp} and references therein.   A representative example here is the Netflix problem, where $A$ is the matrix of ratings (the rows are viewers, the columns are movie titles, and entries are ratings). 
 
 In this section, we are going to use our new results to study this problem.  The main novel feature here is  that our analysis allows us to approximate  {\it any given column (or row)}  with high probability.  For instance, in the  Netflix problem, 
 one can figure out the ratings of any given individual, or any given movie. 
 
  In earlier algorithms we know of, the approximation was mostly  done for the Frobenius norm of the whole matrix.  Such a result is equivalent to saying that a {\it random} row or column is well approximated, but 
  cannot guarantee anything about a specific row or column.  
  
  Finally, let us mention that there are   algorithms which can recover $A$ precisely, but these work only if $A$ satisfies certain structural assumptions \cite{CCS,CP,CR,CRT,CT}. 
 
Without loss of generality,  we assume $A$ is a square  $n \times n$  matrix.  The rectangular case follows by applying the analysis below to an $(n+m) \times (n+m)$  
to the matrix $\tilde{A}$ defined in \eqref{eq:def:tildeA}.  

  Let $A$ be a  $n \times n$ deterministic matrix with rank $r$ where  $\sigma_1 \geq \cdots \geq \sigma_r > 0$ are the singular values  with corresponding singular vectors $u_1, \ldots, u_r$.  
  Let $\chi_{ij}$ be iid indicator random variables with $\Prob (\chi_{ij}=1)=p$.  The entries of the sparse matrix $B$ can be written as 
  \begin{equation*} b_{ij} = (a_{ij } +z_{ij} ) \chi_{ij} = p a_{ij} + a_{ij} (\chi_{ij} -p) + z_{ij} \chi_{ij} = pa_{ij} + f_{ij}, \end{equation*}  where 
  \begin{equation*} f_{ij} :=  a_{ij} (\chi_{ij} -p) + z_{ij} \chi_{ij} . \end{equation*} It is clear that the $f_{ij}$ are independent random variables with mean 0 and variance 
  $\sigma_{ij}^2 =  a_{ij}^2 p(1-p) + p $.  This way, we can write $\frac{1}{p} B$ in the form $A + E$, where  $E$ is the  random matrix with independent entries $e_{ij} := p^{-1} f_{ij}$.  We assume  $p \le 1/2$; in fact, our result 
  works for $p$ being a negative power of $n$.

 Let $1 \le j \le r$ and consider  the subspace $U$ spanned by $u_1, \dots, u_j$ and $V$ spanned by $v_1, \dots, v_j$, where $u_i$ (alternatively $v_i)$  is the $i$-th singular vector of 
 $A$ (alternatively $B$).
 Fix any $1 \le m  \le n$ and 
  consider the $m$-th columns
  of $A$ and $A+E$. Denote them by 
  $x$ and $\tilde x $, respectively.  We have 
  \begin{equation*}  \| x- P_{V} \tilde x\| \le \| x - P_U x\| + \| P_U x - P_U \tilde x\| + \| P_U \tilde x - P_V \tilde x \|. \end{equation*} 
  Notice that  $P_V \tilde x $  is efficiently computable given $B$ and $p$.
  (In fact,  we can estimate $p$ very well by the density of $B$, so we don't even need to know $p$.)  In the remaining part of the analysis, we 
 will estimate the three error terms on the right-hand side.

  We will make use of the following lemma, which is a variant of \cite[Lemma 2.2]{TVdet}; see also \cite{VW} where results of this type are discussed in depth.

 \begin{lemma} \label{lemma:projection} Let $X$ be a random vector in $\R^n$ whose coordinates $x_i, 1\le i \le n$  are independent random variables with mean 0, variance at most $\sigma^2$, and are bounded in absolute value by $1$. Let $H$ be a fixed subspace of dimension $d$ and $P_H (X)$ be the projection of $X$ onto $H$. Then 
 \begin{equation} \label{eqn:distance} \Prob \left(  \| P_H (X) \| \ge \sigma d^{1/2} +  t \right)  \le C  \exp (-c t^2 )  ,\end{equation}  
 where $c, C>0$ are absolute constants. 
 
 \end{lemma}

  The first term $\| x -P_U  x \|$ is  bounded from above by $\sigma_{j+1}$.  The second term has the form $\| P_U X \|$, where $X:=x -\tilde x$ is the random vector
  with independent entries, which is the $m$-th column of   $E$.  Notice that entries of $X$ are bounded (in absolute value) by   $\alpha : = p^{-1} (\|x\| _{\infty}+  K)$ with probability $1$. Applying Lemma \ref{lemma:projection} (with the proper normalization),   we obtain 
\begin{equation} \label{recovery3} 
	\Prob \left(   \| P_U X \| \geq  j^{1/2} \sqrt{ \frac{ \|x\|_{\infty}^2 + 1}{p} }+ t \right)  \le C \exp( -c t^2 \alpha^{-2} ) 
\end{equation} 
since $\sigma_{im}^2 \leq p^{-1} ( \|x\|_{\infty}^2 + 1)$.  
By setting $t := c^{-1/2} \alpha \lambda $,  \eqref{recovery3} implies  that, for any $\lambda > 0$, 
\begin{equation*} \| P_U X \| \le  j^{1/2} \sqrt{ \frac{ \|x\|_{\infty}^2 + 1}{p} }+ c^{-1/2} \lambda \alpha  \end{equation*} 
with probability at least $1-   C \exp(-\lambda^2 ) $.  

To bound  $\| P_U \tilde x - P_V \tilde x \|$, we appeal to Theorem \ref{thm:subspace}. Assume for a moment that $E$ is $(C_1, c_1, \gamma)$-concentrated 
  for some constants $C_1, c_1, \gamma > 0$.  Let $\delta_j := \sigma_j - \sigma_{j+1}$.  Then it follows that, for any $\lambda > 0$, 
\begin{equation*} \| P_U - P_V\| \le C \sqrt{j} \left(\frac{ \lambda^{2/\gamma} r^{1/\gamma} }{ \delta_{j} } + \frac{\|E \| }{\sigma_j} +\frac{\| E\|^2 }{ \sigma_j \delta_j } \right), \end{equation*}  
with probability at least 
$$ 1 - 6C_1 9^j \exp \left( -c_1 \frac{\delta_j^\gamma}{8^\gamma} \right) - 2C_1 9^{2r} \exp \left( -c_1 r \frac{\lambda^2}{4^\gamma} \right), $$
where $C$ is an absolute constant.  
  
Since 
$$ \| P_U \tilde{x} - P_V \tilde{x} \| \leq \|P_U - P_{V} \| \|\tilde{x} \|, $$
it remains to bound $\| \tilde{x} \|$.  We first note that $\|\tilde{x}\| \leq \|x\| + \|X \|$.  By Talagrand's inequality (see \cite{Tconc} or \cite[Theorem 2.1.13]{Tbook}) , we have
$$ \Prob \left( \|X \| \geq \E \|X\| + t \right) \leq C \exp(-c t^2 \alpha^{-2}). $$
In addition, 
$$ \E \|X\|^2 = \frac{1}{p^2} \sum_{i=1}^n \sigma_{im}^2 \leq \frac{1}{p} \left( \|x\|^2 + n \right). $$
Thus, we conclude that
$$ \|X\| \leq \sqrt{ \frac{\|x\|^2 + n }{p} } + c^{-1/2} \lambda \alpha $$
with probability at least $1 - C \exp(-\lambda^2)$.  

Putting the bounds together, we obtain Theorem \ref{theorem:recovery} below.

  \begin{theorem}  \label{theorem:recovery} Assume that $A$ has rank $r$ and $\| Z\|_{\infty} \le K $ with probability $1$. Assume that $E$  is 
  $(C_1, c_1, \gamma)$-concentrated for a trio of constants $C_1, c_1, \gamma >0$.  Let $m$ be an arbitrary index between $1$ and $n$, and let $x$ and $\tilde x$ be the $m$-th columns of $A$ and $\frac{1}{p} B$.  Let $1 \leq j \leq r$ be an integer, and let $V$ be the subspace spanned by the first $j$ singular vectors of $B$.  Let $\sigma_1 \ge \dots \ge \sigma_r > 0$ be the 
  singular values of $A$.   Assume $\delta_j := \sigma_j - \sigma_{j+1}$.  Then, for any $\lambda >0$,  
  \begin{equation*} 
  \| x - P_V (\tilde x) \| \le \sigma_{j+1} +  j^{1/2} \sqrt{ \frac{ \|x\|_{\infty}^2 + 1}{p} } + \mu \left( \sqrt{ \frac{ \|x\|^2 + n}{p} } + C \lambda \alpha \right) + C \lambda \alpha, 
   \end{equation*} 
with probability at least  
$$ 1 - C \exp(-\lambda^2) -  6C_1 9^j \exp \left( -c_1 \frac{\delta_j^\gamma}{8^\gamma} \right) - 2C_1 9^{2r} \exp \left( -c_1 r \frac{\lambda^2}{4^\gamma} \right), $$
  where 
  $$ \alpha := p^{-1} (\| x\| _{\infty} + K) \quad \text{and}\quad \mu:= C  \sqrt{j} \left(\frac{\lambda^{2/\gamma} r^{1/\gamma} }{ \delta_{j} } + \frac{\|E \| }{\sigma_j} +\frac{\| E\|^2 }{ \sigma_j \delta_j } \right), $$
and $C$ is an absolute constant.  
  \end{theorem} 
  
As this theorem is a bit technical, let us consider a special, simpler case. Assume that all entries of $A$ are of order $\Theta(1)$ and $p=\Theta(1)$.   Thus, any column $x$ has length $ \Theta (n^{1/2})$. Assume furthermore that $j=r=\Theta(1)$ and $\sigma_r = \Omega(n^{1/2+\varepsilon})$ for some $\varepsilon > 0$.  Then our analysis yields

\begin{corollary} 
There exists $c_0 > 0$ (depending only on $\varepsilon$) such that, for any given column $x$,  
$$ \| x - P_V (\tilde x) \|  = O(  n^{-c_0} \|x\|) $$
with probability $1-o(1)$. 
\end{corollary}

\subsection*{Acknowledgements}
The authors would like to thank Nicholas Cook and David Renfrew for useful comments.

\end{document}